\documentclass[12pt, reqno]{amsart}
\usepackage{amsmath}
\usepackage{amssymb}
\usepackage{amsfonts}
\usepackage{amsthm}
\usepackage{amsbsy}
\usepackage{geometry}
\usepackage{xspace}
\usepackage{enumitem}
\usepackage{mathrsfs}
\usepackage[utf8]{inputenc}
\usepackage{mathtools}
\usepackage{hyperref}
\usepackage[super]{nth}

\usepackage{comment}
\usepackage[colorinlistoftodos, textsize=footnotesize,textwidth=15ex]{todonotes}

\theoremstyle{plain}
\newtheorem{theorem}{Theorem}[section]
\newtheorem{proposition}[theorem]{Proposition}
\newtheorem{lemma}[theorem]{Lemma}
\newtheorem{corollary}[theorem]{Corollary}
\newtheorem{claim}{Claim}
\newtheorem*{claim*}{Claim}

\theoremstyle{definition}
\newtheorem{convention}[theorem]{Convention}
\newtheorem{definition}[theorem]{Definition}
\newtheorem{example}[theorem]{Example}
\newtheorem{remark}[theorem]{Remark}

\numberwithin{equation}{section}

\renewcommand{\phi}{\varphi}
\renewcommand{\epsilon}{\varepsilon}
\newcommand{\mc}[1]{\mathcal{#1}}

\newcommand{\inv}{^{-1}}
\newcommand{\fl}[1]{\mathscr{#1}}
\newcommand{\tdlc}{tdlc\@\xspace}
\newcommand{\Zb}{\mathbb{Z}}
\newcommand{\Nb}{\mathbb{N}}

\DeclareMathOperator{\id}{id}
\DeclareMathOperator{\Aut}{Aut}
\DeclareMathOperator{\Sym}{Sym}
\DeclareMathOperator{\proj}{proj}
\DeclareMathOperator{\Fix}{Fix}
\DeclareMathOperator{\Stab}{Stab}
\DeclareMathOperator{\conv}{conv}
\DeclareMathOperator{\COS}{\mathsf{CO}}
\DeclareMathOperator{\dist}{dist}

\begin{document}
	
\title{On flat groups in affine buildings}

\author{Sebastian \textsc{Bischof}}
\email{sebastian.bischof@math.uni-paderborn.de}
\address{Justus-Liebig-Universität Gießen, Mathematisches Institut, 35392 Gießen, Germany}

\thanks{Mathematics Subject Classification 2020: 20E42, 22D05}

\begin{abstract}
	In this article we work out the details of flat groups of the automorphism group of locally finite Bruhat--Tits buildings.
\end{abstract}

\maketitle

\section{Introduction}

Buildings have been introduced by Tits in order to study semi-simple algebraic groups from a geometrical point of view. If the building is locally finite, then the automorphism group of that building endowed with the permutation topology becomes a totally disconnected, locally compact (\tdlc) group. The scale and tidy subgroups for certain automorphisms of buildings were studied in \cite{BPR19}. Moreover, the flat rank of the automorphism group of buildings had been computed in \cite{BRW07}.

The purpose of these notes is to see the details for flat subgroups of the automorphism group of an affine building. Willis has shown in \cite[Proposition~$2.7.7$]{willis2025flatgroups} that for each flat group $\fl{H}$ of finite rank and each subgroup $U$ tidy for $\fl{H}$ there exists a decomposition of $U$ into \emph{eigenfactors}. In addition, each such eigenfactor is associated with a \emph{root} $\rho: \fl{H} \to \Zb$, which is a suitable surjective homomorphisms. Willis has already mentioned that the eigenfactors are analogues of root subgroups and that the root homomorphisms are analogues of roots of Cartan subgroups. In this article we will see the analogy in more detail.

\subsection*{Main result}

Let $\Delta$ be an affine building, let $\Sigma$ be an apartment of $\Delta$ and let $\fl{H}$ be the set of all inner automorphisms of translations of the apartment $\Sigma$. We show that the intersection $U$ of the stabilizer of a chamber in the apartment $\Sigma$ with the group of special automorphisms of $\Delta$ is tidy for $\fl{H}$ and each $U$-eigenfactor $E \neq U_{\fl{H}0}$ of $\fl{H}$ is the fixator of some root in $\Sigma$. We discuss in Section \ref{Section: Computing set of roots} a strategy how to compute the set of roots $\Phi(\fl{H})$ of a flat group $\fl{H}$ provided all subgroups $U_{\epsilon}$ in the decomposition in \cite[Proposition~$2.4.2$]{willis2025flatgroups} are $U$-eigenfactors of $\fl{H}$. Using this algorithm we show that for the given flat group $\fl{H}$ the set of roots $\Phi(\fl{H})$ is in one-to-one correspondence with the set of roots of the Coxeter system at infinity. In particular, for $\rho \in \Phi(\fl{H})$ we have $U_{\rho} = \Fix(\gamma)$ for some root $\gamma \in \Phi(W, S)$.

\subsection*{Acknowledgement}

I am very grateful to George Willis for many helpful discussions during an extended research stay in Newcastle. The pleasant atmosphere in his research group has encouraged my work enormously. I thank Jacqui Ramagge for her comments on an earlier draft. The research leading to this article was funded by a fellowship of the DAAD (57556280), by Justus-Liebig-Universit\"at Gie\ss en during the authors PhD studies, and by the DFG Walter Benjamin project BI 2628/1-1.

\section{Preliminaries}

\subsection{Coxeter systems}

Let $W$ be a group and let $S \subseteq W$ be a generating set of elements of order $2$. For $s, t \in S$ we denote the order of $st$ in $W$ by $m_{st}$. Then the pair $(W, S)$ is called \emph{Coxeter system} if the group $W$ admits the presentation
\[ W \cong \langle S \mid (st)^{m_{st}} = 1 \rangle, \]
where there is one relation for each pair $\{s, t\} \subseteq S$ (possibly $s=t$) with $m_{st} < \infty$.

\begin{example}\label{Example: Coxeter system}
	Let $W := \Sym(n+1), s_i := (i \text{ } i+1)$ and $S = \{s_1, \ldots, s_n\}$. Then the pair $(W, S)$ is a Coxeter system.
\end{example}

Let $(W, S)$ be a Coxeter system and let $\ell: W \to \Nb, w \mapsto \min\{ k\in \Nb \mid \exists s_1, \ldots, s_k \in S: w = s_1 \cdots s_k \}$ denote the corresponding length function. The \emph{rank} of a Coxeter system is the cardinality of the set $S$. It is well-known that for $J \subseteq S$ the pair $(\langle J \rangle, J)$ is a Coxeter system (cf.\ \cite[Ch.\ IV, $\S1.8$ Theorem $2$]{Bo02}). 

The \emph{Coxeter diagram} corresponding to $(W, S)$ is the labeled graph $(S, E(S))$ where $E(S) = \{ \{s, t \} \mid m_{st}>2 \}$ and where each edge $\{s,t\}$ is labeled by $m_{st}$ for all $s, t \in S$. We sometimes call the underlying Coxeter diagram the \emph{type} of $(W, S)$. The Coxeter system is called \emph{irreducible} if the underlying Coxeter diagram is connected; otherwise it is called \emph{reducible}.

\begin{remark}
	Suppose that $(W, S)$ is reducible and let $C$ be the vertex set of a connected component of the Coxeter diagram corresponding to $(W, S)$. Then $W \cong \langle C \rangle \times \langle S \backslash C \rangle$ and, hence, $(W, S)$ can be seen as the direct product of two smaller Coxeter systems. This motivates the term \emph{reducible}.
\end{remark}

We call $J \subseteq S$ \emph{spherical} if $\langle J \rangle$ is finite. Given a spherical subset $J\subseteq S$, there exists a unique element of maximal length in $\langle J \rangle$, which we denote by $r_J$ (cf.\ \cite[Corollary $2.19$]{AB08}). The Coxeter system is called \emph{spherical} if $S$ is spherical. If $(W, S)$ is irreducible, then $(W, S)$ is said to be of \emph{affine type} if $W$ is infinite and virtually abelian; equivalently, its underlying Coxeter diagram is of type $\tilde{A}_n$, $\tilde{B}_n$, $\tilde{C}_n$, $\tilde{D}_n$, $\tilde{E}_6$, $\tilde{E}_7$, $\tilde{E}_8$, $\tilde{F}_4$ or $\tilde{G}_2$ (see \cite[Ch.\ VI, $\S 4.3$ Theorem $4$]{Bo02}).

\begin{convention}\label{Convention: (W,S) Coxeter system of finite rank}
	For the rest of this article we let $(W, S)$ be a Coxeter system of finite rank.
\end{convention}

\subsection{Buildings}

A \emph{building of type $(W, S)$} is a pair $\Delta = (\mc{C}, \delta)$ where $\mc{C}$ is a non-empty set and where $\delta: \mc{C} \times \mc{C} \to W$ is a \emph{distance function} satisfying the following axioms, where $x, y\in \mc{C}$ and $w = \delta(x, y)$:
\begin{enumerate}[label=(Bu\arabic*)]
	\item $w = 1_W$ if and only if $x=y$;
	
	\item if $z\in \mc{C}$ satisfies $s := \delta(y, z) \in S$, then $\delta(x, z) \in \{w, ws\}$, and if, furthermore, $\ell(ws) = \ell(w) +1$, then $\delta(x, z) = ws$;
	
	\item if $s\in S$, there exists $z\in \mc{C}$ such that $\delta(y, z) = s$ and $\delta(x, z) = ws$.
\end{enumerate}

The \emph{rank} of $\Delta$ is the rank of the underlying Coxeter system. The building $\Delta$ is called \emph{spherical/affine} if its type is spherical/affine. The elements of $\mc{C}$ are called \emph{chambers}. Given $s\in S$ and $x, y \in \mc{C}$, then $x$ is called \emph{$s$-adjacent} to $y$, if $\delta(x, y) = s$. The chambers $x, y$ are called \emph{adjacent}, if they are $s$-adjacent for some $s\in S$. A \emph{gallery} from $x$ to $y$ is a sequence of chambers $(x = x_0, \ldots, x_k = y)$ such that $x_{i-1}$ and $x_i$ are adjacent for all $1 \leq i \leq k$; the number $k$ is called the \emph{length} of the gallery. A gallery from $x$ to $y$ of length $k$ is called \emph{minimal} if there is no gallery from $x$ to $y$ of length $<k$. In this case we have $\ell(\delta(x, y)) = k$ (cf.\ \cite[Corollary $5.17(1)$]{AB08}).

Given a subset $J \subseteq S$ and $x\in \mc{C}$, the \emph{$J$-residue of $x$} is the set $R_J(x) := \{y \in \mc{C} \mid \delta(x, y) \in \langle J \rangle \}$. Each $J$-residue is a building of type $(\langle J \rangle, J)$ with the distance function induced by $\delta$ (cf.\ \cite[Corollary $5.30$]{AB08}). A \emph{residue} is a subset $R$ of $\mc{C}$ such that there exist $J \subseteq S$ and $x\in \mc{C}$ with $R = R_J(x)$. Since the subset $J$ is uniquely determined by $R$, the set $J$ is called the \emph{type} of $R$ and the \emph{rank} of $R$ is defined to be the cardinality of $J$. Given $x\in \mc{C}$ and a residue $R \subseteq \mc{C}$, then there exists a unique chamber $z\in R$ such that $\ell(\delta(x, y)) = \ell(\delta(x, z)) + \ell(\delta(z, y))$ holds for all $y\in R$. The chamber $z$ is called the \emph{projection of $x$ onto $R$} and is denoted by $\proj_R x$. Moreover, if $z = \proj_R x$ we have $\delta(x, y) = \delta(x, z) \delta(z, y)$ for all $y\in R$ (cf.\ \cite[Proposition $5.34$]{AB08}). A residue is called \emph{spherical} if its type is a spherical subset of $S$. Let $R$ be a spherical residue of type $J$ and let $x, y \in R$. Then $x, y$ are called \emph{opposite in $R$} if $\delta(x, y) = r_J$. A \emph{panel} is a residue of rank $1$. An \emph{$s$-panel} is a panel of type $\{s\}$ for $s\in S$. The building $\Delta$ is called \emph{thick}, if each panel of $\Delta$ contains at least three chambers; $\Delta$ is called \emph{locally finite} if each panel contains only finitely many chambers; $\Delta$ is called \emph{regular} if it is locally finite and if for each $s\in S$ all $s$-panels have the same cardinality. If $\Delta$ is regular and $P$ is an $s$-panel, we define $q_s := \vert P \vert -1$.

Let $\Delta = (\mc{C}, \delta)$ and $\Delta' = (\mc{C}', \delta')$ be two buildings of type $(W, S)$ and let $\sigma:S \to S$ be a bijection satisfying $m_{\sigma(s) \sigma(t)} = m_{st}$ for all $s, t \in S$ (this guarantees that $\sigma$ extends to an automorphism of $W$; cf.\ Exercise~\ref{Exercise: Extension of bijection to automorphism}). A \emph{$\sigma$-isomorphism} from $\mathcal{X} \subseteq \mc{C}$ to $\mathcal{X}' \subseteq \mc{C}'$ is a bijection $\phi: \mathcal{X} \to \mathcal{X}'$ such that for all $x, y \in \mathcal{X}$ and $s\in S$ we have $\delta(x, y) = s$ if and only if $\delta'(\phi(x), \phi(y)) = \sigma(s)$. In this case we call $\mathcal{X}$ and $\mathcal{X}'$ \emph{$\sigma$-isomorphic}. An \emph{isomorphism} between $\mathcal{X}$ and $\mathcal{X}'$ is a $\sigma$-isomorphism for some bijection $\sigma:S \to S$ satisfying $m_{\sigma(s) \sigma(t)} = m_{st}$ for all $s, t \in S$. An \emph{automorphism} of the building $\Delta$ is an isomorphism from $\mc{C}$ to itself. A $\sigma$-isomorphism is called \emph{special} if $\sigma = \id$. We call a set of automorphisms \emph{special} if each automorphism in this set is special.

For a building $\Delta$ we denote the set of all automorphisms of $\Delta$ by $\Aut(\Delta)$ and the set of all special automorphisms by $\Aut_0(\Delta)$. A subgroup $G \leq \Aut_0(\Delta)$ is called \emph{Weyl-transitive} if for all $w\in W$ the group $G$ acts transitive on ordered pairs $(c, d)$ of chambers with $\delta(c, d) = w$. We call a subgroup $G \leq \Aut(\Delta)$ \emph{Weyl-transitive} if $G \cap \Aut_0(\Delta)$ is Weyl-transitive.

\begin{example}\label{Example: building}
	We define $\delta:W \times W \to W, (x, y) \mapsto x^{-1}y$. Then $\Sigma(W, S) := (W, \delta)$ is a building of type $(W, S)$. The group $W$ acts faithfully on $\Sigma(W, S)$ by special automorphisms by multiplication from the left, i.e.\ $W \leq \Aut_0(\Sigma(W, S))$.
\end{example}

\begin{lemma}[Exercise~\ref{Exercise: proj isom commute}]\label{Lemma: projection and isometry commute}
	Let $\Delta = (\mc{C}, \delta)$ be a building of type $(W, S)$ and let $\phi \in \Aut_0(\Delta)$. Then $\phi(\proj_R x) = \proj_{\phi(R)} \phi(x)$ holds for all $x\in \mc{C}$ and each residue $R \subseteq \mc{C}$.
\end{lemma}

\subsection{Apartments}

Let $\Delta = (\mc{C}, \delta)$ be a building of type $(W, S)$. A subset $\Sigma \subseteq \mc{C}$ is called \emph{convex} if for any two chambers $c, d \in \Sigma$ and any minimal gallery $(c_0 = c, \ldots, c_k = d)$, we have $c_i \in \Sigma$ for all $0 \leq i \leq k$. We note that residues are convex (cf.\ \cite[Example~5.44(b)]{AB08}). For $X \subseteq \mc{C}$ we define the \emph{convex hull} $\conv(X)$ of $X$ as the intersection of all convex subsets containing $X$. A subset $\Sigma \subseteq \mc{C}$ is called \emph{thin} if $P \cap \Sigma$ contains exactly two chambers for each panel $P \subseteq \mc{C}$ which meets $\Sigma$. An \emph{apartment} is a non-empty subset $\Sigma \subseteq \mc{C}$, which is convex and thin. It is a basic fact that for an apartment $\Sigma$ and $c\in\Sigma$ the map $\sigma_c: \Sigma \to W, x \mapsto \delta(c, x)$ is a special isomorphism (cf.\ \cite[Proposition $5.65$]{AB08}). Moreover, for any two chambers $c, d \in \mc{C}$ there exists an apartment $\Sigma$ with $c, d \in \Sigma$ (cf.\ \cite[Corollary~5.74]{AB08}).

\begin{lemma}\label{Lemma:gfixesXimpliesgfixesconv(X)}
	Let $\Delta = (\mc{C}, \delta)$ be a building of type $(W, S)$, let $\Sigma$ be an apartment of $\Delta$, let $\phi \in \Aut_0(\Delta)$, and let $X \subseteq \Sigma$ which is fixed pointwise by $\phi$.
	\begin{enumerate}[label=(\alph*)]
		\item Let $\sigma \subseteq \Sigma$ with $\phi(\sigma) = \sigma$. If $X \neq \emptyset$, then $\phi$ fixes $\sigma$ pointwise.
		
		\item $\phi$ fixes $\conv(X)$ pointwise.
	\end{enumerate}
\end{lemma}
\begin{proof}
	(a) Let $\sigma \subseteq \Sigma$ with $\phi(\sigma) = \sigma$, and let $x\in X$. By Exercise~\ref{Exercise: Isometry} the automorphism $\phi$ preserves the distance function. Thus, for $y\in \sigma$ we have $\phi(y) \in \sigma$ and $\delta(x, y) = \delta(\phi(x), \phi(y)) = \delta(x, \phi(y))$. As $x, y, \phi(y) \in \Sigma$, the special isomorphism $\sigma_x$ yields $y = \phi(y)$ and, hence, $\phi$ fixes $\sigma$.
	
	(b) If $X=\emptyset$, then $\conv(X) = \emptyset$ and the claim follows. Thus we can assume $X \neq \emptyset$. Let $C \subseteq \mc{C}$ be any convex set. By \cite[Lemma $5.62$]{AB08} the set $\phi^{\pm 1}(C)$ is convex. We deduce that $\phi^{\pm 1}(\conv(X))$ is convex and contains $X$. Thus $\conv(X) \subseteq \phi^{\pm 1}(\conv(X))$. We conclude $\phi(\conv(X)) \subseteq \conv(X) \subseteq \phi(\conv(X))$ and, hence, $\phi$ stabilizes $\conv(X)$. As $\conv(X) \subseteq \Sigma$ (cf.\ Exercise~\ref{Exercise: Convex sets}), the claim now follows from (a).
\end{proof}

\subsection{Roots}

In this subsection we consider the Coxeter building $\Sigma(W, S)$ from Example~\ref{Example: building}. A \emph{reflection} is an element of $W$ that is conjugate to an element of $S$. For $s\in S$ we let $\alpha_s := \{ w\in W \mid \ell(sw) > \ell(w) \}$ be the \emph{simple root} of $(W, S)$ corresponding to $s$. A \emph{root} of $(W, S)$ is a subset $\alpha \subseteq W$ such that $\alpha = v\alpha_s$ for some $v\in W$ and $s\in S$. We denote the set of all roots of $(W, S)$ by $\Phi := \Phi(W, S)$. Roots can be seen as half-spaces in the Cayley graph of $(W, S)$. For each root $\alpha \in \Phi$ we denote the root \emph{opposite} $\alpha$ by $-\alpha := W \backslash \alpha \in \Phi$ and we denote the unique reflection which interchanges these two roots by $r_{\alpha} \in W \leq \Aut_0(\Sigma(W, S))$. We note that roots are convex (cf.\ \cite[Proposition~5.81(1)]{AB08}). A root $\alpha \in \Phi$ \emph{cuts} a residue $R$ if $\alpha \cap R \neq \emptyset \neq (-\alpha) \cap R$. For a residue $R$ we let $\Phi_R$ be the set of all roots cutting $R$. Two roots $\alpha, \beta \in \Phi$ (not necessarily distinct) are \emph{parallel} if $\alpha \subseteq \beta$ or $\beta \subseteq \alpha$. An element $t\in \Aut(\Sigma(W, S))$ is called \emph{translation} if it maps each root to a parallel root. Note that a translation need not be a special automorphism.

For $\alpha \in \Phi$ we denote by $\partial \alpha$ the set of all panels stabilized by $r_{\alpha}$. The set $\partial \alpha$ is called the \emph{wall} associated with $\alpha$. Let $G = (c_0, \ldots, c_k)$ be a gallery. We say that $G$ \emph{crosses the wall $\partial \alpha$} if there exists $1 \leq i \leq k$ such that $\{ c_{i-1}, c_i \} \in \partial \alpha$.

\begin{lemma}\label{Lemma: minimal gallery crossing roots}
	A gallery is minimal if and only if it crosses each wall at most once.
\end{lemma}
\begin{proof}
	One implication is \cite[Lemma $3.69$]{AB08}. Thus let $G = (c_0, \ldots, c_n)$ be a gallery which crosses each wall at most once. Let $H = (d_0 = c_0, \ldots, d_k = c_n)$ be a minimal gallery. Then by \cite[Lemma $3.69$]{AB08} the $k$ walls crossed by $H$ are distinct and are precisely the walls which \emph{separate} $c_0$ and $c_n$, i.e.\ for each root $\alpha$ with $\vert \alpha \cap \{c_0, c_n\} \vert = 1$ the wall $\partial \alpha$ is crossed by $H$. Thus there are only $k$ such walls. Since the walls crossed by the gallery $G$ are distinct, we infer $n=k$ and the gallery $G$ is minimal.
\end{proof}

\begin{lemma}\label{Lemma:translationinducesminimalgallery}
	Let $t \in \Aut(\Sigma(W, S))$ be a translation and let $(c_0, \ldots, c_k = t(c_0))$ be a minimal gallery. Then, for all $n \in \Nb$, the concatenation of the galleries $(t^i(c_0), \ldots, t^i(c_k))$, $0 \leq i \leq n$, is a minimal gallery.
\end{lemma}
\begin{proof}
	Assume that the gallery is not minimal. By Lemma~\ref{Lemma: minimal gallery crossing roots} the previous lemma there exists a root $\alpha \in \Phi$ such that the wall $\partial \alpha$ is crossed twice by this gallery. Let $m$ be minimal such that $(t^m(c_0), \ldots, t^m(c_k))$ crosses the wall $\partial \alpha$. Then there exists a unique $1 \leq i \leq k$ with $\{t^m(c_{i-1}), t^m(c_i)\} \in \partial \alpha$. Without loss of generality we can assume that $t^m(c_{i-1}) \in \alpha$. Let $l>m$ be minimal such that $(t^l(c_0), \ldots, t^l(c_k))$ crosses the wall $\partial \alpha$. Then there exists a unique $1 \leq j \leq k$ with $\{ t^l(c_{j-1}), t^l(c_j) \} \in \partial \alpha$. Since $\partial \alpha$ is not crossed by $(t^m(c_i), \ldots, t^l(c_{j-1}))$, we infer $t^l(c_j) \in \alpha$.
	
	Note that $\{ t^l(c_{i-1}), t^l(c_i) \} \in \partial t^{l-m}(\alpha)$ and $t^l(c_{i-1}) \in t^{l-m}(\alpha)$. As $t$ is a translation, the roots $\alpha$ and $t^{l-m}(\alpha)$ are parallel and, in particular, $t^{l-m}(\alpha) \neq (-\alpha)$. This implies $i \neq j$. Suppose $i<j$ (resp.\ $i>j$). Since the gallery $(t^l(c_0), \ldots, t^l(c_k))$ is minimal, the walls $\partial \alpha$ and $\partial t^{l-m}(\alpha)$ are crossed exactly once by this gallery by Lemma~\ref{Lemma: minimal gallery crossing roots}. Thus $t^l(c_j) \in \alpha \backslash t^{l-m}(\alpha)$ and $t^l(c_{i-1}) \in t^{l-m}(\alpha) \backslash \alpha$ (resp.\ $t^l(c_{j-1}) \in t^{l-m}(\alpha) \backslash \alpha$ and $t^l(c_i) \in \alpha \backslash t^{l-m}(\alpha)$). Since both roots are parallel, we infer $\alpha = t^{l-m}(\alpha)$, which is a contradiction to the fact $\alpha \backslash t^{l-m}(\alpha) \neq \emptyset$.
\end{proof}

\subsection{Moufang buildings}

Let $\Delta = (\mc{C}, \delta)$ be a building of type $(W, S)$. A subset $\alpha \subseteq \mc{C}$ is called a \emph{root} if it is $\id$-isomorphic to a simple root $\alpha_s \subseteq W$ for some $s\in S$. By \cite[Proposition $5.81(3)$]{AB08} the roots in the building $\Sigma(W, S)$ are precisely of the form $v\alpha_s$ for some $v\in W$ and $s\in S$. Thus we denote the set of all roots in $\Sigma(W, S)$ also by $\Phi$.

\begin{lemma}[{\cite[Proposition~29.20]{We09}}]\label{Lemma:Prop29.20We09}
	Let $X\subseteq W$. Then $\conv(X)$ coincides with the intersection of all roots of $\Sigma(W, S)$ that contain $X$.
\end{lemma}

Let $(W, S)$ be irreducible, spherical of rank at least $2$ and let $\Delta = (\mc{C}, \delta)$ be a thick building of type $(W, S)$. For a root $\alpha \subseteq \mc{C}$ we define the \emph{root group} $U_{\alpha}$ as the set of all automorphisms of $\Delta$ fixing $\alpha$ pointwise and fixing every panel $P$ pointwise, where $\vert P \cap \alpha \vert = 2$. The building $\Delta$ is called a \emph{Moufang} if for every root $\alpha$ of $\Delta$ the root group $U_{\alpha}$ acts simply transitive on the set of apartments containing $\alpha$.

\begin{lemma}\label{Lemma:fixatorofaroot}
	Let $\Delta$ be a spherical Moufang building, let $\Sigma$ be an apartment of $\Delta$, let $\gamma \subseteq \Sigma$ be a root and let $H := \Fix(\Sigma) \subseteq \Aut_0(\Delta)$. Then we have $\Fix(\gamma) = U_{\gamma} H$.
\end{lemma}
\begin{proof}
	By definition, $U_{\gamma}$ and $H$ fix the root $\gamma$. Thus let $g\in \Fix(\gamma)$. We note that since $(W, S)$ is irreducible and of rank at least $2$, \cite[Example $3.128$]{AB08} implies that $g$ is special. Since $\Delta$ is a Moufang building, the group $U_{\gamma}$ acts transitive on the set of apartments containing $\gamma$. Note that $g(\Sigma)$ is an apartment containing $\gamma$. Hence there exists $u \in U_{\gamma}$ with $u(\Sigma) = g(\Sigma)$. In particular, $u^{-1}g \in \Fix(\Sigma) = H$.
\end{proof}

\section{Tdlc groups}\label{Section: Computing set of roots}

\subsection{Recalling definitions}

We first recall some key definitions and results from \cite{willis2025flatgroups} and \cite{Wi04} about \tdlc groups. Let $G$ be a \tdlc group and let $\alpha \in \Aut(G)$. We denote the set of all compact open subgroups of $G$ by $\COS(G)$. Then $U \in \COS(G)$ is called \emph{tidy for $\alpha$} if the following conditions are satisfied, where $U_{\alpha \epsilon} := \bigcap_{n\geq 0} \alpha^{\epsilon n}(U)$ for $\epsilon \in \{+, -\}$:
\begin{enumerate}[label=(T\Alph*)]
	\item $U = U_{\alpha +} U_{\alpha -}$;
	
	\item $\widehat{U}_{\alpha -} := \bigcup_{n\geq 0} \alpha^{-n}(U_{\alpha -})$ is closed.
\end{enumerate}

We note that $U_{\alpha +} = U_{\alpha^{-1} -}$ follows directly from the definition. For $\mathbf{a} \subseteq \Aut(G)$ a finite set of automorphisms of $G$ and $U \in \COS(G)$ we define $U_{\mathbf{a}} := \bigcap\limits_{\alpha \in \mathbf{a}} U_{\alpha +}$.

We say that $U \in \COS(G)$ is \emph{tidy for $g \in G$}, if $U$ is tidy for the inner automorphism $\gamma_g:G \to G, x \mapsto gxg^{-1}$.

\begin{lemma}[{\cite[Corollary~3.5]{Mo02}}]\label{Lemma:Cor3.5Möller}
	Let $g\in G$ and let $U \in \COS(G)$. Then $U$ is tidy for $g$ if and only if $[U : U \cap g^n U g^{-n}] = [U : U \cap gUg^{-1}]^n$ holds for all $n \in \Nb$.
\end{lemma}

Let $\fl{H} \leq \Aut(G)$. Then $U \in \COS(G)$ is said to be \emph{tidy for $\fl{H}$} if $U$ is tidy for all $\alpha \in \fl{H}$. The subgroup $\fl{H}$ is called \emph{flat} if there exists $U \in \COS(G)$ which is tidy for $\fl{H}$. 

Let $\fl{H} \leq \Aut(G)$ be flat and let $U \in \COS(G)$ be tidy for $\fl{H}$. We define $U_{\fl{H}0} := \bigcap \{ \alpha(U) \mid \alpha\in \fl{H} \}$. Moreover, a subgroup $E \leq U$ is called \emph{$U$-eigenfactor of $\fl{H}$} if for each $\alpha \in \fl{H}$ we have $\alpha(E) \leq E$ or $\alpha(E) \geq E$, and, moreover, $E = \bigcap_{\alpha \in \fl{H}, \alpha(E) \geq E} \alpha(U)$.

\begin{lemma}\label{Lemma: UH0 U-eigenfactor}
	Let $\fl{H} \leq \Aut(G)$ be flat and let $U$ be tidy for $\fl{H}$. Then $U_{\fl{H}0}$ is a $U$-eigenfactor of $\fl{H}$.
\end{lemma}
\begin{proof}
	Note that $\alpha(U_{\fl{H}0}) = U_{\fl{H}0}$ for all $\alpha \in \fl{H}$. This implies $\bigcap_{\alpha \in \fl{H}, \alpha(U_{\fl{H}0}) \geq U_{\fl{H}0}} \alpha(U) = \bigcap_{\alpha \in \fl{H}} \alpha(U) = U_{\fl{H}0}$ and the claim follows.
\end{proof}

\begin{lemma}\label{Lemma: U-eigenfactors}
	Let $\fl{H} \leq \Aut(G)$ be flat, let $U$ be tidy for $\fl{H}$ and let $\mathbf{a} \subseteq \fl{H}$ be any subset (not necessarily finite) such that for each $\alpha \in \fl{H}$ we have $\alpha(U_{\mathbf{a}}) \leq U_{\mathbf{a}}$ or $\alpha(U_{\mathbf{a}}) \geq U_{\mathbf{a}}$. Then $U_{\mathbf{a}}$ is a $U$-eigenfactor of $\fl{H}$.
\end{lemma}
\begin{proof}
	By definition it suffices to show that $U_{\mathbf{a}} = \bigcap_{\alpha \in \fl{H}, \alpha(U_{\mathbf{a}}) \geq U_{\mathbf{a}}} \alpha(U)$ holds. One inclusion is obvious. For the other we deduce from \cite[Lemma~$2.4.3$]{willis2025flatgroups} that for all $\alpha \in \mathbf{a}$ we have
	\[ \alpha(U_{\mathbf{a}}) = \alpha\left( \bigcap_{\beta \in \mathbf{a}} (U_{\alpha +} \cap U_{\beta +}) \right) = \bigcap_{\beta \in \mathbf{a}} \alpha(U_{\alpha +} \cap U_{\beta +}) \geq \bigcap_{\beta \in \mathbf{a}} (U_{\alpha +} \cap U_{\beta +}) = U_{\mathbf{a}}. \]
	In particular, $\alpha^n(U_{\mathbf{a}}) \geq U_{\mathbf{a}}$ holds for all $\alpha \in \mathbf{a}$ and $n\in \Nb$. Hence $\bigcap_{\alpha \in \fl{H}, \alpha(U_{\mathbf{a}}) \geq U_{\mathbf{a}}} \alpha(U) \leq U_{\beta +}$ for each $\beta \in \mathbf{a}$. We infer $\bigcap_{\alpha \in \fl{H}, \alpha(U_{\mathbf{a}}) \geq U_{\mathbf{a}}} \alpha(U) \leq U_{\mathbf{a}}$. This finishes the claim.
\end{proof}

We recall that by \cite[Definition~$2.2.1$]{willis2025flatgroups} a closed, non-compact, $\fl{H}$-invariant subgroup $H \leq G$ is called a \emph{scaling subgroup for $\fl{H}$} if there exists a compact, relatively open subgroup $U \leq H$ such that
\begin{enumerate}[label=(\roman*)]
	\item $H = \bigcup \{ \alpha(U) \mid \alpha \in \fl{H} \}$;
	
	\item For each $\alpha \in \fl{H}$ we have $\alpha(U) \geq U$ or $\alpha(U) \leq U$.
\end{enumerate}

\begin{lemma}\label{Lemma: eigenfactor yields unreduced subgroup}
	Let $\fl{H} \leq \Aut(G)$ be flat and let $U$ be tidy for $\fl{H}$. Let $\alpha_1, \ldots, \alpha_k \in \fl{H}$ be such that $E := U_{\{\alpha_1, \ldots, \alpha_k\}}$ is a $U$-eigenfactor of $\fl{H}$. If $E \neq U_{\fl{H}0}$, then $\widehat{E} := \bigcup_{\alpha \in \fl{H}} \alpha(E)$ is a scaling subgroup for $\fl{H}$.
\end{lemma}
\begin{proof}
	By definition, $\widehat{E}$ is $\fl{H}$-invariant and we have $\alpha(E) \leq E$ or $\alpha(E) \geq E$ for each $\alpha \in \fl{H}$.
	
	\setcounter{claim}{0}
	\begin{claim}\label{Claim: E = Ehat cap U}
		We have $E = \widehat{E} \cap U$.
	\end{claim}
	
	One inclusion is obvious. Thus we let $x \in \widehat{E} \cap U$. Then there exists $\alpha \in \fl{H}$ with $x\in \alpha(E)$. As $E$ is closed and a subgroup of $U$, $E$ is compact and $\alpha(E)$ is compact as well. Recall that $\alpha(E) \leq E$ or $\alpha(E) \geq E$. If $\alpha(E) \leq E$, then $x \in E$ and we are done. Thus we can assume that $E \leq \alpha(E)$. Recall that $E = U_{\{\alpha_1, \ldots, \alpha_k\}}$. Using \cite[Lemma~$2.4.3$]{willis2025flatgroups} we obtain $\alpha_i^{-1}(U_{\alpha_i +} \cap U_{\alpha_j +}) \leq U_{\alpha_i +} \cap U_{\alpha_j +}$ for all $1 \leq i\neq j \leq k$. This implies $\alpha_i^{-1}(E) \leq E$ similar as in the proof of Lemma~\ref{Lemma: U-eigenfactors}. Now \cite[Corollary~$2.5.3$]{willis2025flatgroups} and \cite[Corollary~$2.1.4$]{willis2025flatgroups} yield for each $n\in \mathbb{N}$:
	\[ \alpha^{-1} \alpha_i^{-n}(x) = [\alpha, \alpha_i^n] \alpha_i^{-n} \alpha^{-1}(x) \in [\alpha, \alpha_i^n](E) = E. \]
	Hence $\alpha_i^{-n}(x) \in \alpha(E)$ for all $n \in \mathbb{N}$. Since an infinite sequence in a compact set has an accumulation point, we deduce $x \in U_{\alpha_i^{-1} -} =  U_{\alpha_i +}$ by \cite[Proposition~$1.1.6$]{willis2025flatgroups}. Thus $x \in \bigcap_{i=1}^k U_{\alpha_i +} = E$.
	
	\begin{claim}
		$\widehat{E}$ is closed.
	\end{claim}
	
	As $E$ is a $U$-eigenfactor of $\fl{H}$, we know that $\alpha(E)$ contains or is contained in $E$ for all $\alpha \in \fl{H}$. As $E \neq U_{\fl{H}0}$, there exists $\alpha \in \fl{H}$ with $\alpha(E) \lneq E$ and hence $E \lneq \alpha^{-1}(E)$. Suppose $\beta_* \in \fl{H}$ with $\beta_*(E) > E$ and $[\beta_*(E) : E] = \min \{ [\beta(E) : E] \mid \beta \in \fl{H} \text{ and } \beta(E) > E \}$, and let $\alpha \in \fl{H}$. We show that $\alpha(E) = \beta_*^z(E)$ for some $z\in \Zb$. We distinguish the following cases:
	\begin{itemize}[label=$\bullet$]
		\item $\alpha(E) = E$: Then $\alpha(E) = E = \beta_*^0(E)$.
		
		\item $\alpha(E) > E$: As $E$ is a $U$-eigenfactor, we have $\alpha^{-1} \beta_*(E) \geq E$ or $\alpha^{-1} \beta_*(E) \leq E$. Suppose $\alpha^{-1} \beta_*(E) \geq E$ and hence $\beta_*(E) \geq \alpha(E)$. Then we have
		\[ [\beta_*(E) : E] = [\beta_*(E) : \alpha(E)] \cdot [\alpha(E) : E]. \]
		The minimality of $\beta_*$ yields $[\beta_*(E) : \alpha(E)] = 1$ and hence $\beta_*(E) = \alpha(E)$. Thus we can assume $\alpha^{-1} \beta_*(E) \leq E$ and hence $\beta_*(E) \leq \alpha(E)$. If $\beta_*^n(E) \leq \alpha(E)$ for all $n \in \Nb$, then we would have
		\[ [\alpha(E) : E] = [\alpha(E) : \beta_*^n(E)] \prod_{i=0}^{n-1} [\beta_*^{i+1}(E) : \beta_*^i(E)] \geq [\beta_*(E) : E]^{n-1} \geq 2^{n-1}. \]
		This is a contradiction. Thus there exists $n \in \Nb$ with $\beta_*^n(E) \leq \alpha(E) < \beta_*^{n+1}(E)$ (note that $\beta_*^{n+1}(E) \not\leq \alpha(E)$ implies $\alpha(E) < \beta_*^{n+1}(E)$). In particular, $E \leq \beta_*^{-n} \alpha(E) < \beta_*(E)$ and we compute
		\[ [\beta_*(E) : E] = [\beta_*(E) : \beta_*^{-n}\alpha(E)] \cdot [\beta_*^{-n}\alpha(E) : E] > [\beta_*^{-n}\alpha(E) : E]. \]
		The minimality of $\beta_*$ implies $\beta_*^{-n}\alpha(E) = E$. It follows $\alpha(E) = \beta_*^n(E)$.
		
		\item $\alpha(E) < E$: Then $\alpha^{-1}(E) > E$ and we deduce from the previous case $\alpha^{-1}(E) = \beta_*^n(E)$ for some $n\in \Nb$. Note that $[\alpha, \beta_*^n] \in \fl{H}_u$ by \cite[Corollary~$2.5.3$]{willis2025flatgroups}. Using \cite[Corollary~$2.1.4$]{willis2025flatgroups} we infer
		\[ E = [\alpha, \beta_*^n](E) = \alpha^{-1} \beta_*^{-n} \alpha \beta_*^n(E) = \alpha^{-1} \beta_*^{-n}(E). \]
		We deduce $\alpha(E) = \beta_*^{-n}(E)$.
	\end{itemize}
	We conclude $\{ \alpha(E) \mid \alpha \in \fl{H} \} = \{ \beta_*^z(E) \mid z\in \Zb \}$. Thus $\widehat{E}$ is a group because it is a nested union of subgroups of $G$. As in \cite[Proposition~$2.2.5$]{willis2025flatgroups}, $\widehat{E}$ is closed if $\widehat{E} \cap U$ is closed (cf.\ \cite[Theorem II.5.9]{HR79}). As $\widehat{E} \cap U = E$ by Claim~\ref{Claim: E = Ehat cap U} and $E$ is closed, $\widehat{E}$ is closed as well.
	
	As $E$ is compact and relatively open in $\widehat{E}$, it is left to show that $\widehat{E}$ is non-compact. By \cite[Proposition~$1.1.5$]{willis2025flatgroups}, $E$ is tidy for $\fl{H} \vert_{\widehat{E}}$. Thus $s(\beta_* \vert_{\widehat{E}}) > 1$ and $\widehat{E}$ is not compact.
\end{proof}

\begin{lemma}\label{Lemma: flat U decomposition into eigenfactors}
	Let $\fl{H} \leq \Aut(G)$ be flat and $U \in \COS(G)$ be tidy for $\fl{H}$. Let $\alpha_1, \ldots, \alpha_n \in \fl{H}$ and let
	\[ U = \prod_{\epsilon\in\{-,+\}^n} U_\varepsilon \]
	be as in \cite[Proposition~$2.4.2$]{willis2025flatgroups}. Suppose that for each $\alpha \in \fl{H}$ we have $\alpha(U_{\epsilon}) \leq U_{\epsilon}$ or $\alpha(U_{\epsilon}) \geq U_{\epsilon}$. Then $\{ U_{\epsilon}, U_{\fl{H}0} \mid \epsilon \in \{-, +\}^n \}$ is the set of all $U$-eigenfactors of $\fl{H}$.
\end{lemma}
\begin{proof}
	The proof is an adaptation of \cite[Lemma $6.2$]{Wi04}. However, we give the datails here. By Lemma~\ref{Lemma: UH0 U-eigenfactor} and Lemma~\ref{Lemma: U-eigenfactors}, $U_{\fl{H}0}$ and $U_{\epsilon}$ are $U$-eigenfactors of $\fl{H}$. Now let $E$ be any $U$-eigenfactor of $\fl{H}$. Then there exist $\epsilon_i \in \{+, -\}$ with $\alpha_i^{\epsilon_i 1}(E) \geq E$. For this $\epsilon = (\epsilon_1, \ldots, \epsilon_n)$ we have
	\[ E = \bigcap_{\substack{\alpha \in \fl{H} \\ \alpha(E) \geq E}} \alpha(U) \leq U_{\epsilon}. \]
	If $E \neq U_{\epsilon}$, then, since both groups are $U$-eigenfactors of $\fl{H}$, there exists $\alpha \in \fl{H}$ such that $\alpha(E) \geq E$ and $\alpha(U_{\epsilon}) < U_{\epsilon}$. But then $E \leq \bigcap_{k \geq 0} \alpha^k(U_{\epsilon})$. 
	
	\begin{claim*}
		We have $\bigcap_{k \geq 0} \alpha^k(U_{\epsilon}) = U_{\fl{H}0}$.
	\end{claim*}
	
	As $U_{\fl{H}0} = \alpha^k(U_{\fl{H}0}) \leq \alpha^k(U_{\epsilon})$, one inclusion is clear. Thus it suffices to show that $\bigcap_{k \geq 0} \alpha^k(U_{\epsilon}) \subseteq \beta(U)$ for each $\beta \in \fl{H}$. If $\beta(U_{\epsilon}) \geq U_{\epsilon}$, then $U_{\epsilon} \leq \beta(U)$ holds by definition. Thus we can assume $\beta(U_{\epsilon}) < U_{\epsilon}$. Note that by the proof of the previous lemma we have $\alpha(U_{\epsilon}) = \beta_*^{-n}(U_{\epsilon})$ for some $n\in \Nb$ and $\beta(U_{\epsilon}) = \beta_*^{-n'}(U_{\epsilon})$ for some $n' \in \Nb$. In particular, $\beta^n(U_{\epsilon}) = \alpha^{n'}(U_{\epsilon})$. This, together with \cite[Proposition~$1.1.7$]{willis2025flatgroups}, implies
	\[ \bigcap_{k \geq 0} \alpha^k(U_{\epsilon}) \leq U_{\epsilon} \cap \alpha^{n'}(U_{\epsilon}) = U_{\epsilon} \cap \beta^n(U_{\epsilon}) \leq U \cap \beta^n(U) \leq \beta(U). \]
	
	As $U_{\fl{H}0} \leq E$ by definition ($E$ is a $U$-eigenfactor), we conclude that $E = U_{\epsilon}$ or $E = U_{\fl{H}0}$.
\end{proof}

\begin{remark}
	We note that the assumption in the previous lemma is satisfied if $\fl{H}$ is finitely generated (cf.\ \cite[Corollary~5.6]{Wi04}).
\end{remark}

\subsection{Computation of the set of roots $\Phi(\fl{H})$}

In this subsection we let $G$ be a \tdlc group and $\fl{H} \leq \Aut(G)$ be a flat group of automorphisms. Recall from \cite{willis2025flatgroups} that $\fl{H}$ is equipped with a set $\Phi(\fl{H})$ of \emph{roots}. In this section we will provide a strategy how to compute the set of roots $\Phi(\fl{H})$.

\begin{remark}\label{Remark: tidy subgroup of unreduced group}
	Let $\rho \in \Phi(\fl{H})$ be a root. Recall that $\rho: \fl{H} \to \Zb$ is a surjective homomorphism such that there exists a subgroup $H \leq G$, scaling for $\fl{H}$, and $s_{\rho} \in \Nb$, greater than $1$, such that $\Delta_H(\alpha \vert_H) = s_{\rho}^{\rho(\alpha)}$ for every $\alpha \in \fl{H}$, where $\Delta_H$ denotes the modular function of $H$. Let $U \leq H$ be compact and relatively open as in the definition of subgroups scaling for $\fl{H}$ (\cite[Definition~$2.2.1$]{willis2025flatgroups}). We note that in this case we have the following for each $\alpha \in \fl{H}$:
	\[ s_{\rho}^{\rho(\alpha)} = \Delta_H(\alpha \vert_H) = \begin{cases}
		[\alpha(U) : U] & \text{, if } \alpha(U) \geq U \\
		[U: \alpha(U)]^{-1} & \text{, if } \alpha(U) \leq U
	\end{cases} \]
	We note that $\rho(\alpha) \geq 0$ if and only if $\alpha(U) \geq U$, and $\rho(\alpha) \leq 0$ if and only if $\alpha(U) \leq U$. In particular, $\rho(\alpha)=0$ if and only if $\alpha(U) = U$.
	
	Now let $\alpha \in \fl{H}$ and let $V \in \COS(H)$ such that for each $\beta \in \fl{H}$ we have $\beta(V) \leq V$ or $\beta(V) \geq V$ (e.g.\ $V=U$). Then $V$ (and, in particular, $U$) is tidy for $\fl{H} \vert_H$ by \cite[Proposition~$1.1.5$]{willis2025flatgroups}. 
	\begin{enumerate}[label=(\alph*)]
		\item If $\rho(\alpha) >0$, then $\alpha(U) > U$ and $[\alpha(V) : \alpha(V) \cap V] = s(\alpha \vert_H) = [ \alpha(U) : U ] >1$ holds. As $\alpha(V) \cap V \in \{ \alpha(V), V \}$, we deduce $\alpha(V) > V$.
		
		\item If $\alpha(V) = V$, then $1 = [ \alpha^{\pm 1}(V) : \alpha^{\pm 1}(V) \cap V ] = s( \alpha^{\pm 1} \vert_H ) = [ \alpha^{\pm 1}(U) : \alpha^{\pm 1}(U) \cap U ]$. This implies $\alpha^{\pm 1}(U) \cap U = \alpha^{\pm 1}(U)$. We conclude $\alpha(U) = U$ and hence $\rho(\alpha) = 0$.
	\end{enumerate}
\end{remark}

\begin{remark}\label{Remark: flat group generated by kernel}
	Let $\rho \in \Phi(\fl{H})$ be a root. As $\rho: \fl{H} \to \mathbb{Z}$ is surjective, there exists $\alpha \in \fl{H}$ with $\rho(\alpha) = 1$. For $\beta \in \fl{H}$ we have $\rho(\beta) = \rho(\beta)\rho(\alpha) = \rho(\alpha^{\rho(\beta)})$ and hence $\alpha^{-\rho(\beta)} \beta \in \ker(\rho)$. This implies $\fl{H} = \langle \alpha, \ker(\rho) \rangle = \langle \alpha \rangle \ker(\rho)$.
\end{remark}

\begin{lemma}\label{Lemma:rootsofirreduciblefactorscoincide}
	Let $H_1 \leq H_2 \leq G$ be two subgroups scaling for $\fl{H}$ and let $\rho_1, \rho_2 \in \Phi(\fl{H})$ be the corresponding roots. Then $\rho_1 = \rho_2$.
\end{lemma}
\begin{proof}
	Let $U_i \leq H_i$ be compact and relatively open such that $H_i = \bigcup_{\alpha \in \fl{H}} \alpha(U_i)$ and for each $\alpha \in \fl{H}$ we have $\alpha(U_i) \leq U_i$ or $\alpha(U_i) \geq U_i$. Then $V := H_1 \cap U_2$ is compact and relatively open in $H_1$, and we have $\alpha(V) = \alpha(H_1 \cap U_2) = H_1 \cap \alpha(U_2)$. Thus we have $\alpha(V) \leq V$ or $\alpha(V) \geq V$ for each $\alpha \in \fl{H}$.
	
	\setcounter{claim}{0}
	\begin{claim}
		Let $\alpha \in \fl{H}$ with $\rho_1(\alpha)>0$. Then $\rho_2(\alpha) >0$.
	\end{claim}
	
	Let $\alpha \in \fl{H}$ with $\rho_1(\alpha)>0$. Then $\alpha(V) > V$ by Remark \ref{Remark: tidy subgroup of unreduced group} and hence $\alpha(U_2) \cap H_1 = \alpha(V) > V = U_2 \cap H_1$. This implies $\alpha(U_2) > U_2$ and hence $\rho_2(\alpha) >0$. 
	
	\begin{claim}
		$\ker(\rho_2) \subseteq \ker(\rho_1)$.
	\end{claim}
	
	Let $\alpha \in \ker(\rho_2)$. Then $\alpha(U_2) = U_2$ and hence $\alpha(V) = H_1 \cap \alpha(U_2) = H_1 \cap U_2 = V$. It follows from Remark \ref{Remark: tidy subgroup of unreduced group} that $\alpha \in \ker(\rho_1)$.
	
	We now prove the claim. Since $\rho_i$ is a root, there exists $\alpha_i \in \fl{H}$ with $\rho_i(\alpha_i) = 1$. Then $\rho_2(\alpha_1) >0$. By Remark \ref{Remark: flat group generated by kernel} there exist $m \in \Nb$ and $n \in \Zb$ such that $\alpha_1 \in \alpha_2^m \ker(\rho_2)$ and $\alpha_2 \in \alpha_1^n \ker(\rho_1)$. Thus
	\[ \alpha_1 \in \alpha_2^m \ker(\rho_2) \subseteq \alpha_1^{nm} \ker(\rho_1). \]
	It follows $\alpha_1^{1-mn} \in \ker(\rho_1)$ and hence $1-mn = \rho_1(\alpha_1^{1-mn}) = 0$. Thus $m=1=n$. For $\beta \in \fl{H}$ there exist $z\in \Zb$ and $k \in \ker(\rho_2) \subseteq \ker(\rho_1)$ with $\beta = \alpha_2^z k$. Now $\rho_1(\beta) = \rho_1(\alpha_2^z k) = \rho_1(\alpha_1^z) = z = \rho_2(\alpha_2^z k) = \rho_2(\beta)$ and hence $\rho_1 = \rho_2$.
\end{proof}

\begin{theorem}\label{Theorem: roots are associated with eigenfactors}
	Let $U$ be tidy for $\fl{H}$ and let $\rho \in \Phi(\fl{H})$. Suppose $\alpha_1, \ldots, \alpha_k \in \fl{H}$ such that the assumptions in Lemma \ref{Lemma: flat U decomposition into eigenfactors} hold, i.e.\ each $U_{\epsilon}$ satisfies that $\alpha(U_{\epsilon})$ contains or is contained in $U_{\epsilon}$ for all $\alpha \in \fl{H}$. Then there exists a $U$-eigenfactor $E$ of $\fl{H}$ such that $\rho$ is associated with $\bigcup_{\alpha \in \fl{H}} \alpha(E)$.
\end{theorem}
\begin{proof}
	Note that $U_{\rho} = \bigcap \{ \alpha(U) \mid \alpha \in \fl{H}, \, \rho(\alpha) \geq 0 \}$ as in \cite[Proposition~$2.2.5$]{willis2025flatgroups}. Note that $U_{\fl{H}0} < U_{\rho}$, as $\rho$ is a root. If $\rho(\alpha_i) \geq 0$ we have $U_{\rho} \leq U_{\alpha_i +}$; otherwise, we have $U_{\rho} \leq U_{\alpha_i^{-1} +}$. Let $E = U_{\{ \alpha_1^{\epsilon_1}, \ldots, \alpha_k^{\epsilon_k} \}}$ be such that each of the automorphisms $\alpha_i^{\epsilon_i}$ satisfies $\rho(\alpha_i^{\epsilon_i}) \geq 0$. Then $U_{\rho} \leq E$. By definition, we have $\widehat{U}_{\rho} \leq \bigcup_{\alpha \in \fl{H}} \alpha(E)$, and both subgroups are scaling for $\fl{H}$ (cf.\ Lemma~\ref{Lemma: eigenfactor yields unreduced subgroup} and \cite[Proposition~$2.2.5$]{willis2025flatgroups}). Now the two roots associated with these scaling subgroups coincide by Lemma \ref{Lemma:rootsofirreduciblefactorscoincide} and the claim follows.
\end{proof}

\subsection{Topology and locally finite buildings}

Recall from Convention~\ref{Convention: (W,S) Coxeter system of finite rank} that $(W, S)$ is a Coxeter system of finite rank. Let $\Delta = (\mc{C}, \delta)$ be a thick locally finite building of type $(W, S)$. Then the group $\Aut(\Delta)$ equipped with the permutation topology turns it into a \tdlc group. Recall that the family of pointwise stabilizers of finitely many chambers forms a neighbourhood basis of the identity. In particular, pointwise stabilizers of finitely many chambers are compact and open.

\begin{lemma}[Exercise~(\ref{Exercise: compact open subgroup})]\label{Lemma: compact open subgroup}
	Let $G \leq \Aut(\Delta)$ be closed and $c\in \mc{C}$. Then we have $G_c \cap \Aut_0(\Delta) \in \COS(G)$.
\end{lemma}

\begin{remark}
	Let $\Delta = (\mc{C}, \delta)$ be a building of type $(W, S)$. For two galleries $(c_0, \ldots, c_k)$ and $(d_0, \ldots, d_m)$ with $c_k = d_0$, we define the \emph{concatenation} of these galleries by the gallery $(c_0, \ldots, c_k = d_0, \ldots, d_m)$.
\end{remark}

\begin{proposition}\label{Prop:tidystabilizer}
	Let $G \leq \Aut(\Delta)$ be closed and Weyl-transitive. Let $g\in G$, $c\in \mc{C}$ and let $(c_0 := c, \ldots, c_k = g(c))$ be a minimal gallery. Assume that for all $n\in \Nb$ the concatenation of the minimal galleries $(g^i(c_0), \ldots, g^i(c_k))$, $0 \leq i \leq n$, is a minimal gallery. Then $G_c \cap \Aut_0(\Delta)$ is tidy for $g$.
\end{proposition}
\begin{proof}
	For a chamber $d\in \mc{C}$ we abbreviate $G_d^0 := G_d \cap \Aut_0(\Delta)$. By Lemma~\ref{Lemma: compact open subgroup} we have $G_c^0 \in \COS(G)$. Lemma \ref{Lemma:Cor3.5Möller} implies that $G_c^0$ is tidy for $g$ if and only if
	\[ \left[ G_c^0: G_c^0 \cap G_{g^n(c)}^0 \right] = \left[ G_c^0: G_c^0 \cap g^n G_c^0 g^{-n} \right] \overset{!}{=} \left[ G_c^0: G_c^0 \cap g G_c^0 g^{-1} \right]^n = \left[ G_c^0: G_c^0 \cap G_{g(c)}^0 \right]^n \]
	holds for all $n\in \Nb$. Since $g$ need not necessarily be a special automorphism, we cannot apply results from \cite{BPR19}, but we will use techniques from there. We note that $G\cap \Aut_0(\Delta)$ acts transitive on $\mc{C}$, as $G$ is Weyl-transitive. Hence, $\Delta$ is regular. By hypothesis, for all $n \in \Nb$, the concatenation of the minimal galleries $(g^i(c_0), \ldots, g^i(c_k))$, $0 \leq i \leq n$, is a minimal gallery. Let $n\in \Nb$. We denote the minimal gallery from $c$ to $g^{n-1}(c_k) = g^n(c)$ obtained in this way by $(d_0, \ldots, d_{kn})$. Suppose that $g$ is a $\sigma$-isomorphism for some bijection $\sigma:S \to S$ satisfying $m_{\sigma(s) \sigma(t)} = m_{st}$ for all $s, t \in S$. Then $q_s = q_{\sigma(s)}$, and we obtain the following, where the first equation follows from Lemma \ref{Lemma:gfixesXimpliesgfixesconv(X)} together with the facts that $G_c^0$ is special and that any two chambers are contained in a common apartment:
	\begin{align*}
		\left[ G_c^0: G_c^0 \cap G_{g^n(c)}^0 \right] = \left[ G_{d_0}^0: \bigcap\limits_{i=0}^{kn} G_{d_i}^0 \right] = \prod_{i=1}^{kn} \left[ \bigcap\limits_{j=0}^{i-1} G_{d_j}^0: \bigcap\limits_{j=0}^i G_{d_j}^0 \right]
	\end{align*}
	Let $P_i$ be the panel containing $d_{i-1}$ and $d_i$ and let $s_i$ be the type of $P_i$. Then $d_{i-1} = \proj_{P_i} d_0$ (cf.\ Exercise~(\ref{Exercise: projection})). Since $G_c^0$ is special, we have $\left[ \bigcap_{j=0}^{i-1} G_{d_j}^0: \bigcap_{j=0}^i G_{d_j}^0 \right] \leq q_{s_i}$. Let $p \in P_i \setminus \{d_{i-1}\}$. As $\delta(d_0, d_i) = \delta(d_0, d_{i-1}) s_i = \delta(d_0, p)$ and since $G$ is Weyl-transitive, there exists $h\in G \cap \Aut_0(\Delta)$ with $h(d_0) = d_0$ and $h(d_i) = p$. Lemma \ref{Lemma: projection and isometry commute} implies $h(d_{i-1}) = h(\proj_{P_i} d_0) = \proj_{h(P_i)} h(d_0) = \proj_{P_i} d_0 = d_{i-1}$. Now Lemma \ref{Lemma:gfixesXimpliesgfixesconv(X)} yields $h\in \bigcap_{j=0}^{i-1} G_{d_j}^0$. In particular, $\left[ \bigcap_{j=0}^{i-1} G_{d_j}^0: \bigcap_{j=0}^i G_{d_j}^0 \right] = q_{s_i}$. Note that $s_i = \sigma^j(s_i) = s_{kj +i}$ for all $1 \leq i \leq k$ and $0 \leq j \leq n-1$. We infer
	\allowdisplaybreaks
	\begin{align*}
		\prod_{i=1}^{kn} \left[ \bigcap_{j=0}^{i-1} G_{d_j}^0: \bigcap_{j=0}^i G_{d_j}^0 \right] &= \prod_{i=1}^{kn} q_{s_i} \\
		&= \prod_{j=0}^{n-1} q_{\sigma^j(s_1)} \cdots q_{\sigma^j(s_k)} \\
		&= \left( q_{s_1} \cdots q_{s_k} \right)^n \\
		&= \left[ G_c^0: G_c^0 \cap G_{g(c)}^0 \right]^n.
	\end{align*}
	This finishes the claim.
\end{proof}

\begin{corollary}\label{Cor:tidysubgroupfortranslation}
	Let $G \leq \Aut(\Delta)$ be closed and Weyl-transitive, let $\Sigma \subseteq \mc{C}$ be an apartment, let $c\in \Sigma$ and let $t\in \Stab_G(\Sigma)$ be such that $t \vert_{\Sigma} \in \Aut(\Sigma)$ is a translation. Then $G_c \cap \Aut_0(\Delta)$ is tidy for $t$. 
\end{corollary}
\begin{proof}
	This is a consequence of Lemma \ref{Lemma:translationinducesminimalgallery} and Proposition \ref{Prop:tidystabilizer}.
\end{proof}

\begin{remark}
	The tidy subgroup $G_c \cap \mathrm{Aut}_0(\Delta)$ looks somehow mysterious. But it generalizes the case of regular trees which are buildings of type $\tilde{A}_1$. For a translation $t$ along an axis in a tree, the pointwise fixator of an edge on this axis is tidy for $t$ (cf.\ Exercise~$(9)$ in \cite[Section~$3.1$]{willis2025flatgroups}). The edge corresponds to a chamber in the building and the vertices correspond to panels.
\end{remark}

\section{Affine buildings}

\subsection{Sectors}\label{Subsection: Sectors}

This section is based on \cite{We09}.

Let $(W, S)$ be affine of type $\tilde{X}_n$ and let $\Sigma := \Sigma(W, S)$. A vertex of $\tilde{X}_n$ is called \emph{special} if its deletion yields the diagram $X_n$. A \emph{gem} of $\Sigma$ is a residue of type $S \backslash \{o\}$, where $o$ is a special vertex of $\tilde{X}_n$. Let $R$ be a gem and let $c \in R$. We denote by $[R, c]$ the set of all roots $\alpha \in \Phi$ containing $c$ but not some chamber of $R$ adjacent to $c$. The \emph{$(R, c)$-sector} is given by the set of chambers $\sigma(R, c) :=\bigcap_{\alpha \in [R, c]} \alpha$. A \emph{sector} is an $(R, c)$-sector for some gem $R$ and $c\in R$. We relax notation and identify $\Sigma$ with the set of chambers $W$ in the Coxeter building $\Sigma(W, S)$.

\begin{lemma}[{\cite[Proposition~4.3]{We09}}]\label{Lemma:Prop4.3}
	Let $R$ be a gem and $c\in R$. Then $\sigma(R, c) = \{ d \in \Sigma \mid \proj_R d = c \}$.
\end{lemma}

\begin{lemma}[{\cite[Proposition~4.11]{We09}}]\label{Lemma:Prop4.11}
	Let $R$ be a gem, let $\alpha \in \Phi_R$ be a root and let $c \in \alpha \cap R$. Then $\sigma(R, c) \subseteq \alpha$.
\end{lemma}

\begin{lemma}\label{Lemma: root is union of sectors}
	Let $R$ be a gem and let $\alpha \in \Phi_R$. Then $\alpha = \bigcup_{c\in R\cap \alpha} \sigma(R, c)$.
\end{lemma}
\begin{proof}
	One inclusion follows from Lemma~\ref{Lemma:Prop4.11}. For the other let $x\in \alpha$. Since $\alpha \in \Phi_R$, \cite[Lemma~$5.45$]{AB08} yields $\proj_R x \in R \cap \alpha$. Now the claim follows from Lemma~\ref{Lemma:Prop4.3}.
\end{proof}

Let $R$ be a gem, let $c \in R$ be a chamber and let $\alpha \in [R, c]$. We define $T_{R, c, \alpha}$ to be the set of all translations of $\Sigma$ that fix every root in $[R, c]$ except $\alpha$. By \cite[Proposition~1.24]{We09}, the group $T_{R, c, \alpha}$ acts transitively and faithfully on the parallel class of $\alpha$. We let $t_{R, c, \alpha} \in T_{R, c, \alpha}$ be the special translation with $t_{R, c, \alpha}(\alpha) \subsetneq \alpha$ and for each special translation $t \in T_{R, c, \alpha}$ with $t(\alpha) \subsetneq \alpha$ we have $t(\alpha) \subseteq t_{R, c, \alpha}(\alpha)$. Note that the existence of $t_{R, c, \alpha}$ follows from \cite[Proposition~1.28]{We09}: there exists a positive integer $k$ such that for each translation $t$ the translation $t^k$ is a special translation. Moreover, we define the special translation
\[ t_{R, c} := \prod_{\alpha \in [R, c]} t_{R, c, \alpha} \in \Aut_0(\Sigma). \]

\begin{remark}
	Note that the set of all translations of $\Sigma$ is an abelian group by \cite[Proposition~1.24]{We09}. Thus the special translation $t_{R, c}$ is well-defined.
\end{remark}

\begin{lemma}\label{Lemma:oppositeinversetranslation}
	Let $R$ be a gem and let $c, d \in R$ be opposite in $R$. Then $t_{R, d} = t_{R, c}^{-1}$.
\end{lemma}
\begin{proof}
	Using Exercise~(\ref{Exercise: translations}), we have $[R, c] = -[R, d] := \{ -\alpha \mid \alpha \in [R, d] \}$ and  $t_{R, c, \alpha}^{-1} = t_{R, d, -\alpha}$. This implies:
	\[ t_{R, d} = \prod_{\alpha \in [R, d]} t_{R, d, \alpha} = \prod_{\alpha \in [R, c]} t_{R, d, -\alpha} = \prod_{\alpha \in [R, c]} t_{R, c, \alpha}^{-1} = \left( \prod_{\alpha \in [R, c]} t_{R, c, \alpha} \right)^{-1} = t_{R, c}^{-1}. \qedhere  \]
\end{proof}

\begin{lemma}\label{Lemma:invariantsector}
	Let $R$ be a gem and let $c\in R$. Then we have $t_{R, c}(\beta) \subsetneq \beta$ for all $\beta \in [R, c]$. In particular, we have $t_{R, c}(\sigma(R, c)) \subseteq \sigma(R, c)$ and $t_{R, c}(R) \subseteq \sigma(R, c)$.
\end{lemma}
\begin{proof}
	Let $\alpha \in [R, c]$. By definition of the special translation $t_{R, c, \alpha}$ we have $t_{R, c, \alpha}(\beta) = \beta$ for each $\beta \in [R, c]$ with $\alpha \neq \beta$. This implies $t_{R, c}(\alpha) = t_{R, c, \alpha}(\alpha) \subsetneq \alpha$ and we deduce the following:
	\[ t_{R, c}(\sigma(R, c)) = t_{R, c} \left( \bigcap_{\alpha \in [R, c]} \alpha \right) = \bigcap_{\alpha \in [R, c]} t_{R, c}(\alpha) \subsetneq \bigcap_{\alpha \in [R, c]} \alpha = \sigma(R, c). \]
	Let $\alpha \in [R, c]$. Then $R' := t_{R, c}(R)$ is a gem as well and $R' \cap t_{R, c}(\alpha) \neq \emptyset$. By \cite[Proposition~1.13]{We09} we have $\alpha \notin \Phi_{R'}$ and, as $t_{R, c}(\alpha) \subsetneq \alpha$, it follows $R' \subseteq \alpha$. This implies
	\[ t_{R, c}(R) = R' \subseteq \bigcap_{\alpha \in [R, c]} \alpha = \sigma(R, c). \qedhere \]
\end{proof}

\begin{definition}
	Let $\alpha \in \Phi$ be a root and let $R$ be a gem. Then there exists a unique root $\alpha_R \in \Phi_R$ which is parallel to $\alpha$ (see \cite[Proposition~$1.18$]{We09}).
\end{definition}

\begin{lemma}\label{Lemma: alpha_R subseteq alpha}
	Let $R$ be a gem, let $\alpha \in \Phi$ be a root and let $c\in R \cap \alpha$. Then $\alpha_R \subseteq \alpha$.
\end{lemma}
\begin{proof}
	As $c\in R \cap \alpha$, we have $R \cap \alpha \neq \emptyset$. Now there are two cases. If $R \cap (-\alpha) = \emptyset$, then $R \subseteq \alpha$. Moreover, as $\alpha_R$ and $\alpha$ are parallel, the fact $R \cap (-\alpha_R) \neq \emptyset$ implies $\alpha_R \subseteq \alpha$. If $R \cap (-\alpha) \neq \emptyset$, then $\alpha \in \Phi_R$ and, hence, $\alpha_R = \alpha$. This finishes the proof.
\end{proof}

\begin{proposition}\label{Prop:Sectorasconvexhull}
	Let $R$ be a gem and $c\in R$. Then $\sigma(R, c) = \conv(\{ t_{R, c}^n(c) \mid n \in \Nb \})$.
\end{proposition}
\begin{proof}
	We abbreviate $t := t_{R, c}$ and we define $\Lambda := \{ \alpha \in \Phi \mid \forall n\in \Nb: t^n(c) \in \alpha \}$. Using Lemma \ref{Lemma:Prop29.20We09} we have
	\[ \conv(\{ t^n(c) \mid n \in \Nb \}) = \bigcap\limits_{\alpha \in \Lambda} \alpha. \]
	
	Note that Lemma \ref{Lemma:invariantsector} yields $t^n(c) \in \sigma(R, c)$ for all $n\in \mathbb{N}$. Since roots are convex, each sector is a convex set. Hence it suffices to show $\sigma(R, c) \subseteq \alpha$ for each $\alpha \in \Lambda$. 
	
	Let $\alpha \in \Lambda$ be any root. Note that $c\in R \cap \alpha$. By Lemma~\ref{Lemma: alpha_R subseteq alpha} we have $\alpha_R \subseteq \alpha$. We claim that $c\in \alpha_R$. We assume by contrary that $c \notin \alpha_R$, i.e.\ $c\in (-\alpha_R)$. Then Lemma \ref{Lemma:invariantsector} and Lemma \ref{Lemma:Prop4.11} imply $t(R) \subseteq \sigma(R, c) \subseteq (-\alpha_R)$. Since $t(R) \not\subseteq t(-\alpha_R)$ and $t$ is a translation, we deduce $t(-\alpha_R) \subsetneq (-\alpha_R)$ and, in particular, $\alpha_R \subsetneq t(\alpha_R)$ (cf.\ Exercise~(\ref{Exercise: roots})). Now there exists $n\in \Nb$ with $\alpha_R \subseteq \alpha \subseteq t^n(\alpha_R)$. The assumption $c \notin \alpha_R$ implies $t^n(c) \notin t^n(\alpha_R)$ and hence $t^n(c) \notin \alpha$. But this is a contradiction to the assumption $\alpha \in \Lambda$ and we infer $c\in \alpha_R$. 
	
	We conclude that for each $\alpha \in \Lambda$ we have $c \in \alpha_R$. Moreover, Lemma \ref{Lemma:Prop4.11} implies $\sigma(R, c) \subseteq \alpha_R$. In particular, we deduce $\sigma(R, c) \subseteq \alpha_R \subseteq \alpha$ for each $\alpha \in \Lambda$. This finishes the proof.
\end{proof}

\begin{lemma}\label{Lemma:convexhullofsectors}
	Let $R$ be a gem and let $C \subseteq R$. Then the following hold:
	\begin{enumerate}[label=(\alph*)]
		\item $\bigcup_{c\in \conv(C)} \sigma(R, c) = \bigcap_{\alpha \in \Phi_R, \, C \subseteq \alpha} \alpha$;
		
		\item $\conv( \bigcup_{c\in C} \sigma(R, c) ) = \bigcup_{c \in \conv(C)} \sigma(R, c)$.
	\end{enumerate}
\end{lemma}
\begin{proof}
	(a) We abbreviate $\Psi := \{ \alpha \in \Phi_R \mid C \subseteq \alpha \}$. We first assume $\Psi = \emptyset$. This implies that for $\alpha \in \Phi$ with $C \subseteq \alpha$ we have $\alpha \notin \Phi_R$ and hence $R \subseteq \alpha$. We deduce from Lemma~\ref{Lemma:Prop29.20We09} that $R \subseteq \conv(C)$. As $R$ is convex and $C \subseteq R$, we infer $\conv(C) = R$. Now the claim  follows from Lemma \ref{Lemma:Prop4.3}. Thus we can assume $\Psi \neq \emptyset$. Moreover, we can assume that $C \neq \emptyset$, as otherwise the claim follows directly.
	
	Let $\alpha \in \Psi$. Then we have $\conv(C) \subseteq \alpha$. Lemma \ref{Lemma:Prop4.11} yields $\sigma(R, c) \subseteq \alpha$ for each $c\in \conv(C)$. For the reverse inclusion we let $x\in \bigcap_{\alpha \in \Psi} \alpha$. By \cite[Lemma $5.45$]{AB08} we obtain $\proj_R x \in R \cap \bigcap_{\alpha \in \Psi} \alpha$.
		
	\setcounter{claim}{0}
	\begin{claim}\label{Claim: conv C}
		$R \cap \bigcap_{\alpha \in \Psi} \alpha = \conv(C)$
	\end{claim}
	
	We have already seen that $\conv(C) \subseteq R \cap \bigcap_{\alpha \in \Psi} \alpha$. For the other inclusion we show for each $c \notin \conv(C)$ we have $c \notin R \cap \bigcap_{\alpha \in \Psi} \alpha$. Thus let $c \in \Sigma \backslash \conv(C)$. Then by Lemma~\ref{Lemma:Prop29.20We09} there exists a root $\alpha \in \Phi$ containing $C$ but not $c$. If $c\notin R$ we are done. Thus we can assume $c \in R$. We conclude that $\alpha \in \Phi_R$ (as $C \neq \emptyset$) and $c \notin \bigcap_{\alpha \in \Psi} \alpha$.
	
	By Claim~\ref{Claim: conv C} we have $\proj_R x \in \conv(C)$. Now Lemma~\ref{Lemma:Prop4.3} implies $x \in \sigma(R, \proj_R x) \subseteq \bigcup_{c \in \conv(C)} \sigma(R, c)$, as desired.
	
	$(b)$ We let $X := \bigcup_{c\in C} \sigma(R, c)$. Since $\bigcup_{c\in \conv(C)} \sigma(R, c)$ is convex by $(a)$ and contains $X$, one inclusion is trivial. For the other inclusion we recall that by Lemma \ref{Lemma:Prop29.20We09} the convex hull $\conv(X)$ is equal to the intersection of all roots of $(W, S)$ that contain $X$. Thus it suffices to show the following claim:
	
	\begin{claim}
		Let $\alpha \in \Phi$ be a root. If $X \subseteq \alpha$, then $\bigcup_{c\in \conv(C)} \sigma(R, c) \subseteq \alpha$.
	\end{claim}
	
	Let $\alpha \in \Phi$ be a root with $X \subseteq \alpha$. Now we suppose $c\in C$. Then, as $C \subseteq X \subseteq \alpha$, we have $c \in R \cap \alpha$ and Lemma~\ref{Lemma: alpha_R subseteq alpha} implies $\alpha_R \subseteq \alpha$. Note that $\sigma(R, c) \subseteq X \subseteq \alpha$ and \cite[Proposition $4.12$]{We09} yields $c\in \alpha_R$. In particular, we deduce $C \subseteq \alpha_R$ and hence $\conv(C) \subseteq \alpha_R$. Lemma \ref{Lemma:Prop4.11} yields $\sigma(R, d) \subseteq \alpha_R \subseteq \alpha$ for all $d\in \conv(C)$ and the claim follows.
\end{proof}

\begin{lemma}\label{Lemma: convex hull of root and chamber}
	Let $\Sigma$ be an apartment of $\Delta$, let $\gamma \subseteq \Sigma$ be a root and let $c\in \Sigma$. Then $\conv(\gamma \cup \{c\})$ is a root which is parallel to $\gamma$.
\end{lemma}
\begin{proof}
	Let $\delta \subseteq \Sigma$ be a root with $\gamma \cup \{c\} \subseteq \delta$. Then $\conv(\gamma \cup \{c\}) \subseteq \delta$ and, hence, $\conv(\gamma \cup \{c\})\neq \Sigma$. We deduce from \cite[Proposition $4.26$]{We09} that $\conv(\gamma \cup \{c\})$ is a root parallel to $\gamma$.
\end{proof}

\begin{lemma}\label{Lemma:convexhullisroot}
	Let $R$ be a gem, let $c\in R$ be a chamber and let $C \subseteq R$ be a subset. We set $X := \conv(\{ t_{R, e}^n (c) \mid e\in C, n\in \Nb \})$.
	\begin{enumerate}[label=(\alph*)]
		\item For each $d\in C$ we have $\sigma(R, d) \subseteq  X$.
		
		\item If $C \not\subseteq \gamma$ for each root $\gamma \in \Phi_R$, then $X = \Sigma$
		
		\item If $C = R \cap \gamma$ for some root $\gamma \in \Phi_R$, then $X = \conv(\gamma \cup \{c\})$. In particular, $X$ is a root which is parallel to $\gamma$.
	\end{enumerate}
\end{lemma}
\begin{proof}
	(a) Let $d\in C$. It follows from Lemma~\ref{Lemma:invariantsector} that $t_{R, d}(c) \in t_{R, d}(R) \subseteq \sigma(R, d)$. Then Lemma~\ref{Lemma:Prop4.3} implies $\proj_R t_{R, d}(c) = d$ and, hence, $d$ is a chamber on a minimal gallery from $c$ to $t_{R, d}(c)$. In particular, we have $d \in X$. For each $n\in \Nb$, we deduce that $t_{R, d}^n(d)$ is a chamber on a minimal gallery from $t^n_{R, d}(c)$ to $t^{n+1}_{R, d}(c)$. We infer that $t^n_{R, d}(d) \in X$. Using Proposition \ref{Prop:Sectorasconvexhull} and the fact that $X$ is convex, we obtain $\sigma(R, d) = \conv( \{t_{R, d}^n(d) \mid n\in \Nb \}) \subseteq X$.
	
	(b) We suppose that for each root $\gamma \in \Phi_R$ we have $C \not\subseteq \gamma$. This implies $C \neq \emptyset$. Let $\alpha \in \Phi$ be a root containing $C$. Then $R \cap \alpha \neq \emptyset$ and $\alpha \notin \Phi_R$. This implies $R \subseteq \alpha$ and Lemma~\ref{Lemma:Prop29.20We09} yields $R \subseteq \conv(C)$. As $C\subseteq R$ and $R$ is convex, we infer that $R = \conv(C)$. Using Lemma~\ref{Lemma:Prop4.3}, Lemma \ref{Lemma:convexhullofsectors}~$(b)$, part (a) and the fact that $X$ is convex, we obtain
	\[ \Sigma = \bigcup_{d\in R} \sigma(R, d) = \bigcup_{d \in \conv(C)} \sigma(R, d) = \conv\left( \bigcup_{d\in C} \sigma(R, d) \right) \subseteq X \subseteq \Sigma. \]
	
	(c) Let $\gamma \in \Phi_R$ such that $C = R \cap \gamma$. We define $\gamma' := \conv(\gamma \cup \{c\})$ and note that $\gamma'$ is a root parallel to $\gamma$ by Lemma~\ref{Lemma: convex hull of root and chamber}. We need to show that $\gamma' = X$. Note that it follows from Lemma~\ref{Lemma: root is union of sectors} that
	\[ \gamma = \bigcup\nolimits_{d\in C} \sigma(R, d). \]
	Note that $c = t_{R, d}^0(c) \in X$ for any $d\in C$. Using part (a), we conclude that $\{c\} \cup \gamma = \{c\} \cup \bigcup_{d\in C} \sigma(R, d) \subseteq X$. As $X$ is convex, we infer $\gamma' = \conv(\gamma \cup \{c\}) \subseteq X$. For the other inclusion we note that by Lemma~\ref{Lemma:invariantsector} and part (a) we have $t_{R, d}^n(c) \in t_{R, d}^n(R) \subseteq \sigma(R, d) \subseteq \gamma \subseteq \gamma'$ for all $d\in C$ and $1 \leq n \in \Nb$. Since $t_{R, d}^0(c) = c \in \gamma'$ and roots are convex, we infer $X \subseteq \gamma'$. We conclude $X = \gamma'$.
\end{proof}

\subsection{The building at infinity}

A building \emph{of type $\tilde{X}_n$} is a building whose type $(W, S)$ is of type $\tilde{X}_n$. Let $\Delta = (\mc{C}, \delta)$ be a building of type $\tilde{X}_n$. A subset $\sigma \subseteq \mc{C}$ is called a \emph{sector of $\Delta$} if there exists an apartment $\Sigma \subseteq \mc{C}$ such that $\sigma \subseteq \Sigma$ and $\sigma$ is a sector of $\Sigma$. We call two sectors of $\Delta$ \emph{parallel} if their intersection contains a sector.

Following \cite[Ch.\ $8$]{We09} we can attach to $\Delta$ a spherical building $\Delta_{\infty}$ of type $X_n$ \emph{at infinity} (with respect to the complete system of apartments) whose chambers are parallel classes of sectors. For an apartment $\Sigma$ of $\Delta$, the set $\Sigma_{\infty}$ of chambers of $\Delta_{\infty}$ consisting of all parallel classes of sectors in $\Sigma$ forms an apartment of the building $\Delta_{\infty}$ (cf.\ \cite[Proposition $8.22$]{We09}). Moreover, for each root $\gamma \subseteq \Sigma$, the set $\gamma_{\infty}$ of chambers of $\Delta_{\infty}$ consisting of all parallel classes of sectors in $\gamma$ forms a root of $\Sigma_{\infty}$ (cf.\ \cite[Proposition $8.30(i)$]{We09}). By construction of $\Delta_{\infty}$ every automorphism of $\Delta$ induces an automorphism on $\Delta_{\infty}$. For more information we refer to \cite{We09}.

\begin{lemma}\label{Lemma:fixatorofarootfixesrootatinfty}
	Let $\Delta$ be a building of type $\tilde{X}_n$ and let $\Delta_{\infty}$ be the corresponding spherical building at infinity. Let $\gamma$ be a root of $\Delta$ and let $\gamma_{\infty}$ be the corresponding root of $\Delta_{\infty}$. Then $\Fix(\gamma) \subseteq \Fix(\gamma_{\infty})$.
\end{lemma}
\begin{proof}
	Let $g \in \Fix(\gamma)$ and let $\sigma \in \gamma_{\infty}$ be a parallel class of sectors. By definition, there exists a sector $S \in \sigma$ which is contained in $\gamma$. Since $g$ fixes $\gamma$ pointwise, the induced map on the building $\Delta_{\infty}$ fixes $\sigma$. Thus $g$ fixes any chamber of the root $\gamma_{\infty}$.
\end{proof}

\subsection{Bruhat--Tits buildings}\label{Subsection: Bruhat--Tits buildings}

A \emph{Bruhat--Tits building} is an affine building $\Delta$ of type $\tilde{X}_n$ with $n \geq 2$ whose spherical building $\Delta_{\infty}$ at infinity of type $X_n$ is a Moufang building. Since all thick, irreducible spherical buildings of rank at least three are Moufang buildings it follows that all thick, irreducible, affine buildings of rank at least four are Bruhat--Tits buildings (cf.\ \cite[Proposition $11.3$]{We09} and \cite[Theorem $11.6$]{We03}).

\begin{remark}
	Not all thick, irreducible affine buildings of rank $3$ (i.e.\ of type $\tilde{A}_2$, $\tilde{B}_2$ or $\tilde{G}_2$) are Bruhat--Tits buildings. Affine buildings whose building at infinity is not a Moufang building are called \emph{exotic}.
\end{remark}

Let $\Delta = (\mc{C}, \delta)$ be a Bruhat--Tits building of type $\tilde{X}_n$. Let $\alpha$ be a root of $\Delta_{\infty}$ and let $U_{\alpha} \leq \Aut_0(\Delta_{\infty})$ be the corresponding root group. For $G^{\dagger} := \langle U_{\alpha} \mid \alpha \text{ root of } \Delta_{\infty} \rangle \leq \Aut_0(\Delta_{\infty})$ we have $G^{\dagger} \leq \Aut_0(\Delta)$ (cf.\ \cite[Proposition~13.2]{We09}).

Let $\Sigma \subseteq \mc{C}$ be an apartment. Let $(W_{\infty}, S_{\infty})$ be the Coxeter system of type $X_n$ and let $\Phi_{\infty} := \Phi(W_{\infty}, S_{\infty})$. For $a\in \Phi_{\infty}$ a root of $\Sigma_{\infty}$ and $u \in U_a^* \leq \Aut_0(\Delta)$ we let $a_u := \Sigma \cap u(\Sigma) \subseteq \mc{C}$. By \cite[Proposition~13.2]{We09} $a_u$ is a root of $\Sigma$ with $(a_u)_{\infty} = a$.

Let $R$ be a residue of $\Delta$ of type $X_n$ which cuts $\Sigma$. In particular, $R\cap \Sigma$ is a gem of $\Sigma$. For each root $a\in \Phi_{\infty}$ of $\Sigma_{\infty}$, we let $a_R$ be the unique root of $\Sigma$ with $(a_R)_{\infty} = a$ and $a_R \in \Phi_{R\cap \Sigma}$ (cf.\ \cite[Proposition~1.18 and Proposition~8.30(ii)]{We09}). By \cite[Proposition $1.3$]{We09} there exists for each root $\alpha \in \Phi$ a bijection $i \mapsto \alpha_i$ from $\Zb$ to the parallel class of $\alpha$ such that $\alpha_0 = \alpha$ and $\alpha_i \subsetneq \alpha_{i+1}$. We put $\mathrm{dist}(\alpha_i, \alpha_j) := j-i$ and define
\[ \phi_a: U_a^* \to \Zb, u \mapsto \dist(a_R, a_u). \]
This mapping is surjective by \cite[Proposition~$13.5$]{We09}. We extend $\phi_a$ to $U_a$ by defining $\phi_a(1) := \infty$ and we set for each $k \in \Zb$
\[ U_{a, k} := \{ u\in U_a \mid \phi_a(u) \geq k \}. \]

\begin{proposition}\label{Prop:stabilizerofaroot}
	Let $\Delta$ be a Bruhat--Tits building of type $\tilde{X}_n$, let $G^{\dagger} \leq F \leq \Aut(\Delta)$, let $\Sigma$ be an apartment of $\Delta$, let $\gamma \subseteq \Sigma$ be a root and let $R$ be a residue of type $X_n$ which cuts $\Sigma$. Let $k = \dist((\gamma_{\infty})_R, \gamma)$. Then $\Fix_F(\gamma) = U_{\gamma_{\infty}, k} \Fix_F(\Sigma)$.
\end{proposition}
\begin{proof}
	Clearly, we have $U_{\gamma_{\infty}, k}, \Fix_F(\Sigma) \leq \Fix_F(\gamma)$. For the reverse inclusion we note that $\Fix_F(\gamma) \subseteq \Fix(\gamma) \subseteq \Fix(\gamma_{\infty})$ by Lemma \ref{Lemma:fixatorofarootfixesrootatinfty}. Using Lemma~\ref{Lemma:fixatorofaroot} we obtain $\Fix_F(\gamma) \subseteq U_{\gamma_{\infty}} H \cap F$, where $H := \Fix(\Sigma_{\infty})$. Now let $g \in \Fix_F(\gamma)$. Then there exist $u \in U_{\gamma_{\infty}}$ and $h\in H$ with $g = uh$. The element $u^{-1}g$ fixes a root in $\Sigma$ parallel to $\gamma$, say $\gamma'$, by \cite[Proposition~13.2]{We09}. In particular, $u^{-1}g \in \Aut_0(\Delta)$. Now $h = u\inv g \in H \cap F$ and $h(\Sigma_{\infty}) = \Sigma_{\infty}$. We also have $h(\Sigma)_{\infty} = h(\Sigma_{\infty}) = \Sigma_{\infty}$ (as $h(\sigma)_{\infty} = h(\sigma_{\infty})$, where $\sigma_{\infty}$ denotes the equivalence class of the sector $\sigma$). But then $h(\Sigma)_{\infty} = \Sigma_{\infty}$ and \cite[Proposition~8.27]{We09} now implies $h(\Sigma) = \Sigma$. As $h = u^{-1}g \in \Stab_F(\Sigma) \cap \Aut_0(\Delta) \cap \Fix(\gamma')$. Now Lemma~\ref{Lemma:gfixesXimpliesgfixesconv(X)} implies $u^{-1} g \in \Fix_F(\Sigma)$ and hence $g = uh \in U_{\gamma_{\infty}} \Fix_F(\Sigma)$.
	
	We conclude that $u = gh^{-1} \in \Fix_F(\gamma)$. Now it follows from \cite[Notation $13.17$]{We09} that $\Fix_F(\gamma) \cap U_{\gamma_{\infty}} = \Fix_{U_{\gamma_{\infty}}}(\gamma) = U_{\gamma_{\infty}, k}$.
\end{proof}

\subsection{Locally finite Bruhat--Tits buildings}

\subsubsection*{Setting}

In this subsection we let $\Delta = (\mc{C}, \delta)$ be a thick locally finite Bruhat--Tits building of type $\tilde{X}_n$. Let $\Delta_{\infty}$ be the Moufang building of type $X_n$ at infinity and let $G^{\dagger} := \langle U_{\alpha} \mid \alpha \text{ root of } \Delta_{\infty} \rangle \leq \Aut_0(\Delta_{\infty})$. We have already discussed that $G^{\dagger} \leq \Aut_0(\Delta)$.

Let $F \leq \Aut(\Delta)$ be a closed subgroup containing $G^{\dagger}$. By \cite[Proposition $18.4$]{We09} and \cite[Corollary $6.12$]{AB08}, the group $F$ is Weyl-transitive. Let $\Sigma$ be an apartment of $\Delta$ and let $c\in \Sigma$. We let $H := \{ g\in \Stab_F(\Sigma) \mid g \vert_{\Sigma} \text{ translation} \} \leq F$. Moreover, we let $\fl{H} := \{ \gamma_g \mid g\in H \} \leq \Aut(F)$, where $\gamma_g: F\to F, x \mapsto gxg^{-1}$. By Corollary~\ref{Cor:tidysubgroupfortranslation}, $U := F_c \cap \Aut_0(\Delta)$ is tidy for $\fl{H}$ and $\fl{H}$ is flat.

\subsubsection*{Decomposition of $U$ into $U$-eigenfactors of $\fl{H}$}

Now we will decompose the tidy subgroup $U$ into a finite product of $U$-eigenfactors of $\fl{H}$ as in Lemma~\ref{Lemma: flat U decomposition into eigenfactors}. Let $N := \Stab_F(\Sigma) \cap \Aut_0(\Delta)$. Note that by \cite[2.8]{We03} we have $\Aut_0(\Sigma) = W$ (where $(W, S)$ is the type of $\Delta$). As $G^{\dagger} \leq F$, \cite[Proposition~18.3(ii)]{We09} implies $W \leq N \vert_{\Sigma}$. Let $R$ be a gem of $\Sigma$ and let $d\in R$. Recall from Subsection \ref{Subsection: Sectors} that we have defined the special translation
\[ t_{R, d} := \prod_{\alpha \in [R, d]} t_{R, d, \alpha} \in \Aut_0(\Sigma) = W. \]
Now we choose for each $\alpha \in [R, d]$ an element $f_{R, d, \alpha} \in N$ with $f_{R, d, \alpha} \vert_{\Sigma} = t_{R, d, \alpha} \in \Aut_0(\Sigma)$. We note that $f_{R, d, \alpha}$ need not be unique. We want to define $f_{R, d}$ similarly as $t_{R, d}$, but the elements $f_{R, d, \alpha}$ may not commute. Thus we have to fix an order on the set of roots $[R, d]$ in advance. Using this order, we define
\[ f_{R, d} := \prod_{\alpha \in [R, d]} f_{R, d, \alpha} \in N. \]

\begin{remark}\label{Remark: fRd special translation}
	As $f_{R, d, \alpha} \vert_{\Sigma} = t_{R, d, \alpha}$, we deduce that $f_{R, d, \alpha} \vert_{\Sigma}$ is a special translation. In particular, $f_{R, d} \vert_{\Sigma}$ is a special translation. Moreover, we have $f_{R, d} \vert_{\Sigma} = t_{R, d}$.
\end{remark}

We abbreviate $f_d := f_{R, d}$ for all $d\in R$. Since $f_d \in H$ we obtain $\gamma_{(f_d)} \in \fl{H}$. 

\begin{convention}\label{Convention: gamma_i}
	For the rest of this section we let $R$ be a gem of $\Sigma$ containing the already chosen chamber $c$, let $d$ be the unique chamber in $R$ opposite to $c$, and let $(c_0 = c, \ldots, c_k = d)$ be a minimal gallery. We abbreviate $\gamma_i := \gamma_{(f_{c_i})}$.
\end{convention}

\begin{definition}
	For any root $\gamma \in \Phi_R$ we define $\gamma' := \conv(\gamma \cup \{c\})$. Recall that $\gamma'$ is a root of $\Sigma$ which is parallel to $\gamma$.
\end{definition}

\begin{lemma}\label{Lemma: U gamma is fixator of apartment or root}
	\begin{enumerate}[label=(\alph*)]
		\item For all $\epsilon_1, \ldots, \epsilon_k \in \{\pm 1\}$ the following hold:
		\[ U_{\{\gamma_1^{\epsilon_1}, \ldots, \gamma_k^{\epsilon_k}\}} \in \{ \Fix_F(\Sigma), \Fix_F(\gamma') \mid \gamma \in \Phi_R \}. \]
		
		\item For each $\gamma \in \Phi_R$ there exist $\epsilon_1, \ldots, \epsilon_k \in \{\pm 1\}$ such that $U_{\{\gamma_1^{\epsilon_1}, \ldots, \gamma_k^{\epsilon_k}\}} = \Fix_F(\gamma')$.
	\end{enumerate}
\end{lemma}
\begin{proof}
	By definition we have $U_{\left\{ \gamma_1^{\epsilon_1}, \ldots, \gamma_k^{\epsilon_k} \right\}} = \bigcap_{i=1}^k U_{\gamma_i^{\epsilon_i} +} = \Aut_0(\Delta) \cap \bigcap_{i=1}^k (F_c)_{\gamma_i^{\epsilon_i} +}$. For each $1 \leq i \leq k$ we let $e_i \in \Sigma$ be the unique chamber in $R$ which is opposite to $c_i$. We define
	\[ d_i := \begin{cases}
		c_i, & \epsilon_i = 1, \\
		e_i, & \epsilon_i = -1,
	\end{cases} \qquad \text{and} \qquad \delta_i := \gamma_{(f_{d_i})}. \]
	Note that $\delta_i = \gamma_i$ if $\epsilon_i = 1$. Now Lemma~\ref{Lemma:oppositeinversetranslation} implies for $\epsilon_i = -1$:
	\allowdisplaybreaks
	\begin{align*}
		(F_c)_{\gamma_i^{-1} +} &= \bigcap_{n \geq 0} \left(\gamma\inv_{(f_{c_i})}\right)^n(F_c) = \bigcap_{n \geq 0} \left(\gamma_{(f_{c_i}^{-1})}\right)^n(F_c) = \bigcap_{n\geq 0} F_{f_{c_i}^{-n}(c)} \\
		&\overset{\ref{Lemma:oppositeinversetranslation}}{=} \bigcap_{n\geq 0} F_{f_{d_i}^n(c)} = \bigcap_{n \geq 0} \left(\gamma_{(f_{d_i})}\right)^n(F_c) = (F_c)_{\gamma_{(f_{d_i})} +}.
	\end{align*}
	This yields $U_{\left\{ \gamma_1^{\epsilon_1}, \ldots, \gamma_k^{\epsilon_k} \right\}} = U_{\{ \delta_1, \ldots, \delta_k\} }$. For each $1 \leq i \leq k$ the group $(F_c)_{\delta_i +}$ is the pointwise stabilizer of the chambers $f_{d_i}^n(c) = t_{R, d_i}^n(c)$ (see Remark~\ref{Remark: fRd special translation}) for all $n\in \Nb$. By Lemma \ref{Lemma:gfixesXimpliesgfixesconv(X)} we have
	\[ U_{\{\delta_1, \ldots, \delta_k\}} = \Fix_F( \conv( \{ t_{R, d_i}^n(c) \mid n \in \Nb, \, 1 \leq i \leq k \} ) ) \cap \Aut_0(\Delta). \]
	For $C := \{ d_1, \ldots, d_k \} \subseteq R$ we have the following two possibilities.
	\begin{itemize}[label=$\bullet$]
		\item For each $\gamma \in \Phi_R$ we have $C \not\subseteq \gamma$. Then Lemma \ref{Lemma:convexhullisroot}$(b)$ yields $U_{\{\delta_1, \ldots, \delta_k\}} = \Fix_F(\Sigma) \cap \Aut_0(\Delta) = \Fix_F(\Sigma)$.
		
		\item There exists $\gamma \in \Phi_R$ with $C \subseteq \gamma$. In this case we have $C \subseteq R \cap \gamma$ and \cite[Proposition~29.18]{We09} yields $C = R \cap \gamma$. Then we have $U_{\{\delta_1, \ldots, \delta_k\}} = \Fix_F(\gamma') \cap \Aut_0(\Delta) = \Fix_F(\gamma')$ by Lemma \ref{Lemma:convexhullisroot}$(c)$, where $\gamma' := \conv(\gamma \cup \{c\})$.
	\end{itemize}
	This shows (a). For (b) let $\gamma \in \Phi_R$. For each $1 \leq i \leq k$ define
	\[ \epsilon_i := \begin{cases}
		1, & c_i \in \gamma, \\
		-1, & c_i \notin \gamma.
	\end{cases} \]
	With the notation as above, we get $C = \{d_1, \ldots, d_k\} \subseteq \gamma$, hence $C = \gamma \cap R$. It follows $U_{\{\gamma_1^{\epsilon_1}, \ldots, \gamma_k^{\epsilon_k}\}} = \Fix_F(\gamma')$.
\end{proof}

\begin{theorem}\label{Theorem:Eigenfactorfixesatleastaroot}
	For all $\epsilon_1, \ldots, \epsilon_k \in \{\pm 1\}$ the subgroup $U_{\left\{ \gamma_1^{\epsilon_1}, \ldots, \gamma_k^{\epsilon_k} \right\}}$ is a $U$-eigenfactor of $\fl{H}$.
\end{theorem}
\begin{proof}
	We abbreviate $E := U_{\left\{ \gamma_1^{\epsilon_1}, \ldots, \gamma_k^{\epsilon_k} \right\}}$. By Lemma~\ref{Lemma: U gamma is fixator of apartment or root} we have $E = \Fix_F(\Sigma)$ or $E = \Fix_F(\gamma')$, where $\gamma' = \conv(\gamma \cup \{c\})$ for some root $\gamma \in \Phi_R$. Note that $\gamma'$ is a root as well. To show that $E$ is a $U$-eigenfactor, it suffices by Lemma~\ref{Lemma: U-eigenfactors} to show that for each $\alpha \in \fl{H}$ we have $\alpha(E) \leq E$ or $\alpha(E) \geq E$. Recall that $\fl{H} = \{ \gamma_g \mid g\in H \}$, and let $g\in H$. We consider the following two cases:
	\begin{enumerate}[label=(\alph*)]
		\item $E = \Fix_F(\Sigma)$: Then $\gamma_g(\Fix_F(\Sigma)) = \Fix_F(g(\Sigma)) = \Fix_F(\Sigma)$ and we are done.
		
		\item $E = \Fix_F(\gamma')$: Then $\gamma_g(\Fix_F(\gamma')) = \Fix_F(g(\gamma'))$ holds. As $g \vert_{\Sigma}$ is a translation, $g(\gamma')$ is a root parallel to $\gamma'$. This finishes the claim. \qedhere
	\end{enumerate}
\end{proof}

\begin{corollary}
	Suppose $\epsilon_1, \ldots, \epsilon_k \in \{\pm 1\}$ with $U_{\left\{ \gamma_1^{\epsilon_1}, \ldots, \gamma_k^{\epsilon_k} \right\}} = \Fix_F(\gamma')$ for some root $\gamma' \in \Phi$. Then
	\[ \bigcup_{\alpha \in \fl{H}} \alpha \left(U_{\left\{ \gamma_1^{\epsilon_1}, \ldots, \gamma_k^{\epsilon_k} \right\}} \right) = U_{(\gamma')_{\infty}} \Fix_F(\Sigma). \]
\end{corollary}
\begin{proof}
	Let $t\in H$. As $t(\gamma')$ is a root parallel to $\gamma'$, we obtain $(\gamma')_{\infty} = t(\gamma')_{\infty}$ by \cite[Proposition~4.22]{We09}. Now Proposition~\ref{Prop:stabilizerofaroot} implies \[ \gamma_t(\Fix_F(\gamma')) = \Fix_F(t(\gamma')) \subseteq U_{t(\gamma')_{\infty}} \Fix_F(\Sigma) = U_{(\gamma')_{\infty}} \Fix_F(\Sigma). \]
	This shows one inclusion. For the other we note that $\Fix_F(\Sigma) \leq \Fix_F(\gamma')$. Hence, it suffices to consider $U_{(\gamma')_{\infty}}$. Thus let $u \in U_{(\gamma')_{\infty}} \setminus\{1\}$. We obtain by \cite[Proposition~13.2]{We09} that $\beta := \Sigma \cap u(\Sigma)$ is a root with $\beta_{\infty} = (\gamma')_{\infty}$ on which $u$ acts trivially. In particular, $\beta$ and $\gamma'$ are parallel. If $\gamma' \subseteq \beta$, we infer $u \in \Fix_F(\gamma')$ and we are done. Thus we can assume $\beta \subsetneq \gamma'$. By \cite[Corollary~1.26 and Proposition~1.28]{We09} there exists a special translation $t \in \Aut_0(\Sigma) = W$ with $t(\gamma') \subseteq \beta$. Then there exists $g \in H$ with $g \vert_{\Sigma} = t$. This implies $u \in \Fix_F(\beta) \leq \Fix_F\left(g(\gamma')\right) = \gamma_g(\Fix_F(\gamma'))$. This finishes the claim.
\end{proof}

\subsubsection*{Computation of $\Phi(\fl{H})$}

We apply \cite[Proposition~$2.4.2$]{willis2025flatgroups} to $U$ and the automorphisms $\gamma_1, \ldots, \gamma_k \in \fl{H}$ from Convention~\ref{Convention: gamma_i}. By Theorem~\ref{Theorem:Eigenfactorfixesatleastaroot} all the subgroups $U_{\epsilon}$ are $U$-eigenfactors of $\fl{H}$. Lemma~\ref{Lemma: flat U decomposition into eigenfactors} implies that $\{ U_{\fl{H}0}, U_{\epsilon} \mid \epsilon \in \{+, - \}^n \}$ is the set of all $U$-eigenfactors of $\fl{H}$. For a root $\gamma \in \Phi_R$ we define $\gamma' := \conv(\gamma \cup \{c\})$ and recall that $\gamma'$ is a root as well.

\begin{lemma}\label{Lemma: PhiR bijection to U-eigenfactors}
	The map from $\Phi_R$ to the set of all $U$-eigenfactors of $\fl{H}$ different from $U_{\fl{H}0}$ which maps $\gamma$ to $\Fix_F(\gamma')$ is a bijection.
\end{lemma}
\begin{proof}
	By Lemma~\ref{Lemma: U gamma is fixator of apartment or root}(a) we know that each $U$-eigenfactor different from $U_{\fl{H}0} = \Fix_F(\Sigma)$ is equal to $\Fix_F(\gamma')$ for some root $\gamma \in \Phi_R$. Moreover, by Lemma~\ref{Lemma: U gamma is fixator of apartment or root}(b), for each root $\gamma \in \Phi_R$, $\Fix_F(\gamma')$ is indeed a $U$-eigenfactor of $\fl{H}$ (different from $U_{\fl{H}0})$. Thus the map is well-defined and surjective. Now let $\gamma_1 \neq \gamma_2 \in \Phi_R$. It is left to show that $\Fix_F(\gamma_1') \neq \Fix_F(\gamma_2')$. Note that by \cite[Proposition~$1.15$]{We09} one of the following hold:
	\begin{itemize}
		\item $\gamma_1'$ and $\gamma_2'$ are parallel: As parallelism of roots of $\Sigma$ is an equivalence relation by \cite[Corollary~$1.4$]{We09}, the roots $\gamma_1$ and $\gamma_2$ are parallel as well. Then \cite[Proposition~$1.13$]{We09} implies $\gamma_1 = \gamma_2$, which is a contradiction.
		
		\item $\gamma_1'$ and $(-\gamma_2')$ are parallel: Then again $\gamma_1$ and $(-\gamma_2)$ are parallel and we have $\gamma_1 = -\gamma_2$. Note that $\gamma_i \subseteq \gamma_i'$. Thus $\Fix_F(\gamma_1') = \Fix_F(\gamma_2')$ would imply 
		\[ \Fix_F(\gamma_1') = \Fix_F(\gamma_2') \leq \Fix_F(\gamma_1) \cap \Fix_F(\gamma_2) = \Fix_F(\Sigma), \]
		which is a contradiction.
		
		\item There exists a gem $R'$ with $\gamma_1', \gamma_2' \in \Phi_{R'}$. If $\gamma_1' \neq \gamma_2'$, then $\Fix_F(\gamma_1') \neq \Fix_F(\gamma_2')$. Thus we can assume $\gamma_1' = \gamma_2'$. As parallelism of roots of $\Sigma$ is an equivalence relation, the roots $\gamma_1$ and $\gamma_2$ are parallel. But then \cite[Proposition~$1.13$]{We09} implies $\gamma_1 = \gamma_2$, which is a contradiction. \qedhere
	\end{itemize}
\end{proof}

\begin{definition}
	Let $\gamma \in \Phi_R$ be a root. By Theorem~\ref{Theorem:Eigenfactorfixesatleastaroot}, $E := \Fix_F(\gamma')$ is a $U$-eigenfactor of $\fl{H}$. By Lemma~\ref{Lemma: eigenfactor yields unreduced subgroup},
	\[ \widehat{E} = \bigcup\nolimits_{\alpha \in \fl{H}} \alpha(\Fix_F(\gamma')) \]
	is a scaling subgroup for $\fl{H}$. By \cite[Definition~$2.2.3$]{willis2025flatgroups} there exists a root $\rho_{\gamma} \in \Phi(\fl{H})$ which is associated to $\widehat{E}$. Moreover, for $\alpha \in \fl{H}$ we have
	\begin{align*}
		\rho_{\gamma}(\alpha) & \geq 0 \Longleftrightarrow \alpha(\Fix_F(\gamma')) \geq \Fix_F(\gamma') \qquad \text{and} \\
		\rho_{\gamma}(\alpha) &\leq 0 \Longleftrightarrow \alpha(\Fix_F(\gamma')) \leq \Fix_F(\gamma').
	\end{align*}
\end{definition}

\begin{proposition}
	The mapping $\phi: \Phi_R \to \Phi(\fl{H}), \gamma \mapsto \rho_{\gamma}$ is a well-defined bijection.
\end{proposition}
\begin{proof}
	By Lemma~\ref{Lemma: flat U decomposition into eigenfactors} and Theorem~\ref{Theorem: roots are associated with eigenfactors} all roots $\rho \in \Phi(\fl{H})$ are of the form $\rho_{\gamma}$ for some $\gamma \in \Phi_R$. Thus it suffices to show that $\phi$ is injective. Let $\gamma, \delta \in \Phi_R$ with $\gamma \neq \delta$. Then there exists a translation $t$ such that $t(\gamma') \supsetneq \gamma'$ and $t(\delta') \subsetneq \delta'$: The case $\gamma = -\delta$ is obvious and the case $\delta \neq -\gamma$ follows essentially from \cite[Proposition~1.24]{We09}. This implies
	\allowdisplaybreaks
	\begin{align*}
		\gamma_t(\Fix_F(\gamma')) = \Fix_F(t(\gamma')) < \Fix_F(\gamma') &\Longrightarrow \rho_{\gamma}(\gamma_t) < 0, \\
		\gamma_t(\Fix_F(\delta')) = \Fix_F(t(\delta')) > \Fix_F(\delta') &\Longrightarrow \rho_{\delta}(\gamma_t) >0.
	\end{align*}
	We conclude $\rho_{\gamma} \neq \rho_{\delta}$ and $\phi$ is injective.
\end{proof}

\begin{remark}
	Recall that $R$ is a gem. Hence $R$ is a residue of type $X_n$ by definition. Note that $(W_{\infty}, S_{\infty})$ is also of type $X_n$. Thus there is a bijection from $\Phi(W_{\infty}, S_{\infty})$ to $\Phi_R$. The number of roots of such a Coxeter system of type $X_n$ is explicitly stated in \cite[(I) in Plate I-IX]{Bo02}
	\[ \begin{tabular}{|c|c|c|c|c|c|c|c|c|}
		\hline type & $A_n$ & $B_n/C_n$ & $D_n$ & $E_6$ & $E_7$ & $E_8$ & $F_4$ & $G_2$ \\
		\hline number of roots & $n(n+1)$ & $2n^2$ & $2n(n-1)$ & $72$ & $126$ & $240$ & $48$ & $12$ \\
		\hline
	\end{tabular} \]
\end{remark}

\subsubsection*{Explicit values of roots in $\Phi(\fl{H})$}

Let $\gamma \in \Phi_R$ and let $\rho_{\gamma}$ be the corresponding root. Define 
\[ s_{\gamma} := \min\{ [\beta(\Fix_F(\gamma')): \Fix_F(\gamma')] \mid \beta \in \fl{H}, \beta(\Fix_F(\gamma')) > \Fix_F(\gamma') \}. \]
Recall from Remark~\ref{Remark: tidy subgroup of unreduced group} that the root $\rho_{\gamma}: \fl{H} \to \Zb$ satisfies
\[ s_{\gamma}^{\rho_{\gamma}(\alpha)} = \begin{cases}
	[\alpha(\Fix_F(\gamma')) : \Fix_F(\gamma')] & \text{, if } \alpha(\Fix_F(\gamma')) \geq \Fix_F(\gamma') \\
	[\Fix_F(\gamma'): \alpha(\Fix_F(\gamma'))]^{-1} & \text{, if } \alpha(\Fix_F(\gamma')) \leq \Fix_F(\gamma')
\end{cases} \]
Note that $\gamma_t(\Fix_F(\gamma')) = \Fix_F(t(\gamma'))$. Thus we obtain the following explicit values of the root $\rho_{\gamma}$, where $t\in H$:
\allowdisplaybreaks
\begin{align*}
	\rho_{\gamma}(\gamma_t) &= \begin{cases}
		\log_{s_{\gamma}}([ \Fix_F(t(\gamma')): \Fix_F(\gamma')]) & \text{, if } t(\gamma') \subseteq \gamma' \\
		-\log_{s_{\gamma}}([\Fix_F(\gamma'): \Fix_F(t(\gamma'))]) & \text{, if } t(\gamma') \supseteq \gamma'
	\end{cases}
\end{align*}

\section{Equivalent definition of roots}

\begin{convention}
	In this section we let $G$ be a \tdlc group and $\fl{H} \leq \Aut(G)$ be flat.
\end{convention}

\begin{definition}\label{Definition: roots Bischof}
	A surjective homomorphism $\rho: \fl{H} \to \Zb$ is called a \emph{root} if for each subgroup $U \in \COS(G)$ tidy for $\fl{H}$ we have $U_{\fl{H}0} < U_{\rho}$.
\end{definition}

\begin{remark}
	Definition~\ref{Definition: roots Bischof} is closer to the definition of roots of Cartan subgroups than \cite[Definition~$2.2.3$]{willis2025flatgroups}. If we call $U_{\rho}$ the \emph{root space} of $\rho$, then a surjective homomorphism is a root, if its root space is non-trivial, which in the present situation means that it strictly contains $U_{\fl{H}0}$. However, we will see in Theorem~\ref{Theorem: Equivalent definitions of roots} that the definition of roots in Definition~\ref{Definition: roots Bischof} is equivalent to the definition of roots in \cite[Definition~$2.2.3$]{willis2025flatgroups}.
\end{remark}

We first want to show that a surjective homomorhism $\rho$ is a root if and only if $U_{\fl{H}0} < U_{\rho}$ holds for some $U \in \COS(G)$ tidy for $\fl{H}$. Therefore, we need a few auxiliary results.

\begin{lemma}\label{Lemma: root implies non-compact and closed}
	Let $\rho: \fl{H} \to \Zb$ be a surjective homomorphism and let $U \in \COS(G)$ be tidy for $\fl{H}$. Then the following hold:
	\begin{enumerate}[label=(\alph*)]
		\item If $V\in \COS(G)$ is tidy for $\fl{H}$ and contained in $U$, then $\widehat{U}_{\rho} \cap V = V_{\rho}$;
		
		\item $\widehat{U}_{\rho}$ is closed.
		
		\item $U_{\fl{H}0} < U_{\rho}$ if and only if $\widehat{U}_{\rho}$ is non-compact.
	\end{enumerate}
\end{lemma}
\begin{proof}
	One inclusion in part $(a)$ is clear. Thus let $x \in \widehat{U}_{\rho} \cap V$. Then $x\in \alpha(U_{\rho})$ for some $\alpha \in \fl{H}$. Let $\beta \in \fl{H}$ with $\rho(\beta) \geq 0$. Now \cite[Corollary~$2.5.3$]{willis2025flatgroups} and \cite[Corollary~$2.1.4$]{willis2025flatgroups} yield for each $n\in \mathbb{N}$:
	\[ \alpha^{-1} \beta^{-n}(x) = [ \alpha, \beta^n ] \beta^{-n} \alpha^{-1}(x) \in [ \alpha, \beta^n ](U_{\rho}) = U_{\rho}. \]
	For the last equation we note that for each $\gamma \in \fl{H}$ the subgroup $\gamma(U)$ is tidy for $\fl{H}$ and, as $[\alpha, \beta^n] \in \fl{H}_{\mathrm{u}}$, we have $[ \alpha, \beta^n ](\gamma(U)) = \gamma(U)$. This yields $[ \alpha, \beta^n ](U_{\rho}) = U_{\rho}$.
	
	This implies $\beta^{-n}(x) \in \alpha(U_{\rho})$ for all $n\geq 0$. As $\alpha(U_{\rho})$ is compact and $V$ is tidy for $\fl{H}$, we deduce from \cite[Proposition~$1.1.6$]{willis2025flatgroups} that $x\in V_{\beta^{-1} -} = V_{\beta +} \subseteq \beta(V)$. As this holds for any $\beta \in \fl{H}$ with $\rho(\beta) \geq 0$, we conclude that $x\in V_{\rho}$ This proves $(a)$.
	
	Recall that $\alpha(U_{\rho}) = \beta(U_{\rho})$ for $\alpha, \beta \in \fl{H}$ with $\rho(\alpha) = \rho(\beta)$. Let $\beta \in \fl{H}$ with $\rho(\beta) = 1$. Then $\alpha(U_{\rho}) = \beta^z(U_{\rho})$ for some $z\in \Zb$. In particular, $\widehat{U}_{\rho} = \bigcup_{z\in \Zb} \beta^z(U_{\rho})$ is group because it is a nested union of subgroups of $G$. As in \cite[Proposition~$2.2.5$]{willis2025flatgroups}, $\widehat{U}_{\rho}$ is closed if $\widehat{U}_{\rho} \cap U$ is closed (cf.\ \cite[Theorem II.5.9]{HR79}). But this follows from $(a)$.
	
	Suppose first that $\widehat{U}_{\rho}$ is non-compact. Then $U_{\fl{H}0} = U_{\rho}$ would imply that $\widehat{U}_{\rho} = U_{\fl{H}0}$ is compact which is a contradiction. Thus we can assume $U_{\fl{H}0} < U_{\rho}$. Note that $U_{\rho}$ is tidy for $\fl{H} \vert_{\widehat{U}_{\rho}}$ by \cite[Proposition~$1.1.5$]{willis2025flatgroups}. As $U_{\fl{H}0} < U_{\rho}$, there exists $\alpha \in \fl{H}$ with $U_{\rho} \not\leq \alpha(U)$ and, hence, $\alpha^{-1}(U_{\rho}) \not\leq U$. In particular, we have $\alpha\inv(U_{\rho}) \not\leq U_{\rho}$, hence $\alpha^{-1}(U_{\rho}) > U_{\rho}$. This implies $s(\alpha^{-1} \vert_{\widehat{U}_{\rho}}) >1$ and $\widehat{U}_{\rho}$ is non-compact. 
\end{proof}

\begin{theorem}\label{Theorem: homomorphisms and roots}
	Let $U, V\in \COS(G)$ be tidy for $\fl{H}$ and let $\rho: \fl{H} \to \Zb$ be a surjective homomorphism. Then the following holds:
	\[ U_{\fl{H}0} < U_{\rho} \Longleftrightarrow V_{\fl{H}0} < V_{\rho}. \]
\end{theorem}
\begin{proof}
	Note that $U\cap V$ is again tidy for $\fl{H}$ by \cite[Theorem~$1.1.8$]{willis2025flatgroups}. Thus it suffices to show the claim for $V \leq U$. 
	
	Suppose first that $V_{\fl{H}0} < V_{\rho}$. Then Lemma~\ref{Lemma: root implies non-compact and closed} implies that $\widehat{V}_{\rho}$ is closed and non-compact, and that $\widehat{U}_{\rho} \cap V = \widehat{V}_{\rho}$. By definition, we have $\widehat{V}_{\rho} \leq \widehat{U}_{\rho}$. Now suppose that $U_{\fl{H}0} = U_{\rho}$. Then $\widehat{U}_{\rho} = U_{\fl{H}0}$ would be compact and $\widehat{V}_{\rho}$ would be compact as well (as a closed subgroup of a compact group). This yields a contradiction and we have $U_{\fl{H}0} < U_{\rho}$.
	
	Now suppose that $U_{\fl{H}0} < U_{\rho}$. By Lemma~\ref{Lemma: root implies non-compact and closed} the subgroup $\widehat{U}_{\rho}$ is closed and non-compact, and we have $\widehat{U}_{\rho} \cap V = V_{\rho}$. As $\widehat{U}_{\rho}$ is non-compact, there exists $\beta \in \fl{H}$ with $\beta(U_{\rho}) > U_{\rho}$. As $U_{\rho}$ is tidy for $\fl{H} \vert_{\widehat{U}_{\rho}}$, we have $s( \beta \vert_{\widehat{U}_{\rho}} ) = [\beta(U_{\rho}): \beta(U_{\rho}) \cap U_{\rho}] = [\beta(U_{\rho}): U_{\rho}] >1$. Note that $\alpha(V_{\rho})$ contains or is contained in $V_{\rho}$ for all $\alpha \in \fl{H}$. As $V_{\rho} = \widehat{U}_{\rho} \cap V$ is compact and open in $\widehat{U}_{\rho}$, $V_{\rho} \leq \widehat{U}_{\rho}$ is tidy for $\fl{H} \vert_{\widehat{U}_{\rho}}$ by \cite[Proposition~$1.1.5$]{willis2025flatgroups}. We infer
	\[ 1 < s(\beta \vert_{\widehat{U}_{\rho}}) = [\beta(V_{\rho}) : \beta(V_{\rho}) \cap V_{\rho}]. \]
	It follows $\beta(V_{\rho}) > V_{\rho}$. As $\beta(V_{\fl{H}0}) = V_{\fl{H}0}$, we deduce $V_{\fl{H}0} < V_{\rho}$ and the claim follows.
\end{proof}

\begin{corollary}
	Let $\rho: \fl{H} \to \Zb$ be a surjective homomorphism and let $U \in \COS(G)$ be tidy for $\fl{H}$. Then the following are equivalent:
	\begin{enumerate}[label=(\roman*)]
		\item $\rho$ is a root.
		
		\item $U_{\fl{H}0} < U_{\rho}$.
	\end{enumerate}
\end{corollary}
\begin{proof}
	This follows from Theorem~\ref{Theorem: homomorphisms and roots}.
\end{proof}

\begin{lemma}\label{Lemma: s_ell well-defined}
	Let $\rho: \fl{H} \to \Zb$ be a surjective homomorphism and let $U, V \in \COS(G)$ be tidy for $\fl{H}$. Then
	\[ \min\{ [\beta(U_{\rho}) : U_{\rho}] \mid \beta \in \fl{H}, \, \beta(U_{\rho}) > U_{\rho} \} = \min\{ [\beta(V_{\rho}) : V_{\rho}] \mid \beta \in \fl{H}, \, \beta(V_{\rho}) > V_{\rho} \}. \]
\end{lemma}
\begin{proof}
	As $U\cap V$ is again tidy for $\fl{H}$ by \cite[Theorem~$1.1.8$]{willis2025flatgroups}, it suffices to show the claim for $V \leq U$. Note that $U_{\rho}$ and $V_{\rho}$ are tidy for $\fl{H} \vert_{\widehat{U}_{\rho}}$. Let $\beta \in \fl{H}$. Then
	\begin{equation}\label{Equation: scale}
		[\beta(U_{\rho}) : \beta(U_{\rho}) \cap U_{\rho}] = s( \beta \vert_{\widehat{U}_{\rho}} ) = [\beta(V_{\rho}) : \beta(V_{\rho}) \cap V_{\rho}].
	\end{equation}
	Moreover, we know that $\beta(U_{\rho})$ (resp.\ $\beta(V_{\rho})$) contains or is contained in $U_{\rho}$ (resp.\ $V_{\rho}$).	Now (\ref{Equation: scale}) implies that $\beta(U_{\rho}) > U_{\rho}$ if and only if $\beta(V_{\rho}) > V_{\rho}$. Using (\ref{Equation: scale}) again, the claim follows.
\end{proof}

\begin{definition}\label{Definition: s_ell}
	Let $\rho: \fl{H} \to \Zb$ be a root in the sense of Definition~\ref{Definition: roots Bischof} and let $U \in \COS(G)$ be tidy for $\fl{H}$. Then we define
	\[ s_{\rho} := \min\{ [\beta(U_{\rho}) : U_{\rho}] \mid \beta \in \fl{H}, \, \beta(U_{\rho}) > U_{\rho} \}. \]
	This is well-defined by Lemma~\ref{Lemma: s_ell well-defined}. Moreover, we have $s_{\rho} >1$ by Lemma~\ref{Lemma: root implies non-compact and closed}.
\end{definition}

\begin{theorem}\label{Theorem: Equivalent definitions of roots}
	Let $\rho: \fl{H} \to \Zb$ be a homomorphism. Then the following are equivalent:
	\begin{enumerate}[label=(\roman*)]
		\item $\rho$ is a root in the sense of \cite[Definition~$2.2.3$]{willis2025flatgroups}.
		
		\item $\rho$ is a root in the sense of Definition~\ref{Definition: roots Bischof}.
	\end{enumerate}
\end{theorem}
\begin{proof}
	Assume that $\rho$ is a root in the sense of \cite[Definition~$2.2.3$]{willis2025flatgroups} and let $U \in \COS(G)$ be tidy for $\fl{H}$. Then \cite[Proposition~$2.2.5$]{willis2025flatgroups} implies that $\widehat{U}_{\rho}$ is a scaling subgroup for $\fl{H}$. In particular, we have $U_{\fl{H}0} < U_{\rho}$ (as otherwise $\widehat{U}_{\rho} = U_{\fl{H}0}$ would be compact) and $\rho$ is a root in the sense of Definition~\ref{Definition: roots Bischof}.
	
	Now suppose that $\rho$ is a root in the sense of Definition~\ref{Definition: roots Bischof} and let $U \in \COS(G)$ be tidy for $\fl{H}$. Then $\widehat{U}_{\rho}$ is closed and non-compact, and we have $\widehat{U}_{\rho} \cap U = U_{\rho}$ by Lemma~\ref{Lemma: root implies non-compact and closed}. For $\alpha \in \fl{H}$ we have $\alpha(U_{\rho}) \geq U_{\rho}$ or $\alpha(U_{\rho}) \leq U_{\rho}$. Thus $\widehat{U}_{\rho}$ is scaling for $\fl{H}$. To see that $\rho$ is a root in the sense of \cite[Definition~$2.2.3$]{willis2025flatgroups} it suffices to show that there exists $s_{\rho} \in \Nb$, greater than $1$, with
	\[ \Delta_{\widehat{U}_{\rho}}(\alpha \vert_{\widehat{U}_{\rho}} ) = s_{\rho}^{\rho(\alpha)} \quad \text{for every $\alpha \in \fl{H}$}. \]
	Recall that $U_{\rho}$ is tidy for $\fl{H} \vert_{\widehat{U}_{\rho}}$ and \cite[Theorem~$1.1.10$]{willis2025flatgroups} implies
	\[ \Delta_{\widehat{U}_{\rho}}(\alpha \vert_{\widehat{U}_{\rho}} ) = s(\alpha \vert_{\widehat{U}_{\rho}}) / s(\alpha^{-1} \vert_{\widehat{U}_{\rho}}) = \begin{cases}
		[ \alpha(U_{\rho}) : U_{\rho} ] & \text{, } \alpha(U_{\rho}) \geq U_{\rho} \\
		[ U_{\rho} : \alpha(U_{\rho}) ]^{-1} & \text{, } \alpha(U_{\rho}) \leq U_{\rho}
	\end{cases} \]
	Define $s_{\rho} \in \Nb$ as in Definition~\ref{Definition: s_ell}. Note that $\alpha \in \fl{H}$ with $\rho(\alpha) = 0$ satisfies $\alpha(U_{\rho}) = U_{\rho}$. Recall that for $\alpha, \beta \in \fl{H}$ with $\rho(\alpha) = \rho(\beta)$ we have $\alpha(U_{\rho}) = \beta(U_{\rho})$. Let $\beta_* \in \fl{H}$ be such that $\rho(\beta_*) = 1$. Then for $\alpha \in \fl{H}$ there exists $z\in \Zb$ with $\rho(\alpha) = \rho(\beta_*^z) = z$. We conclude
	\[ \Delta_{\widehat{U}_{\rho}}(\alpha \vert_{\widehat{U}_{\rho}} ) = \begin{cases}
		[ \alpha(U_{\rho}) : U_{\rho} ] & \text{, } \alpha(U_{\rho}) \geq U_{\rho} \\
		[ U_{\rho} : \alpha(U_{\rho}) ]^{-1} & \text{, } \alpha(U_{\rho}) \leq U_{\rho}
	\end{cases} = \begin{cases}
		s_{\rho}^z & \text{, } \alpha(U_{\rho}) \geq U_{\rho} \\
		(s_{\rho}^{-z})^{-1} & \text{, } \alpha(U_{\rho}) \leq U_{\rho}
	\end{cases} = s_{\rho}^{\rho(\alpha)}. \qedhere \]
\end{proof}

\newpage
\section{Exercises}\label{sec:Bischof_exercises}

\begin{enumerate}
	\item Prove that the pair $(W, S)$ from Example~\ref{Example: Coxeter system} is indeed a Coxeter system.
	
	\item\label{Exercise: Extension of bijection to automorphism} Let $(W, S)$ be a Coxeter system and let $\sigma: S \to S$ a bijection with $m_{\sigma(s) \sigma(t)} = m_{st}$ for all $s, t \in S$. Show that $\sigma$ extends to an automorphism of $W$.
	
	\item\label{Exercise: roots} Let $(W, S)$ be a Coxeter system, let $\phi \in \Aut(\Sigma(W, S))$ be an automorphism and let $\alpha \in \Phi$ be a root. Show that $-\phi(\alpha) = \phi(-\alpha)$ holds.
	
	\item Let $(W, S)$ be a Coxeter system. Show that $\Sigma(W, S)$ as defined in Example~\ref{Example: building} is a building of type $(W, S)$. Show also that $W$ acts on $\Sigma(W, S)$ by multiplication from the left.
	
	\item\label{Exercise: proj isom commute} Prove Lemma \ref{Lemma: projection and isometry commute}. What happens, if $\phi$ is not special?
	
	\item\label{Exercise: Convex sets} Let $\Delta = (\mc{C}, \delta)$ be a building of type $(W, S)$. Let $A \subseteq B$ be subsets of $\mc{C}$.
	\begin{enumerate}
		\item Show that $\conv(A) \subseteq \conv(B)$.
		
		\item Show that if $A$ is convex, then $\conv(A) = A$.
	\end{enumerate}
	
	\item\label{Exercise: Isometry} Let $\Delta = (\mc{C}, \delta)$ be a building of type $(W, S)$ and let $\phi \in \Aut_0(\Delta)$. Thow that $\phi$ preserves $\delta$, i.e.\ $\delta(\phi(c), \phi(d)) = \delta(c, d)$ for all $c, d \in \mc{C}$.
	
	\item\label{Exercise: Computation of set of roots} Compute the set of roots $\Phi(W, S)$ for $\ldots$
	\begin{enumerate}
		\item $\ldots$ $(W, S)$ as in Example~\ref{Example: Coxeter system}.
		
		\item $\ldots$ $W = D_{\infty} \cong \langle s, t \mid s^2 = t^2 = 1 \rangle$ and $S = \{ s, t \}$.
	\end{enumerate}
	
	\item\label{Exercise: compact open subgroup} Prove Lemma \ref{Lemma: compact open subgroup}.
	
	\item\label{Exercise: projection} Let $\Delta = (\mc{C}, \delta)$ be a building of type $(W, S)$ and let $(d_0, \ldots, d_k)$ be a minimal gallery. Let $1 \leq i \leq k$ and let $P_i$ be the panel containing $d_{i-1}$ and $d_i$. Then $\proj_{P_i} d_0 = d_{i-1}$.
	
	\item\label{Exercise: translations} Let $\Delta = (\mc{C}, \delta)$ be a building of type $\tilde{X}_n$, let $R$ be a gem and let $c, d \in R$ be opposite in $R$. Show the following:
	\begin{enumerate}
		\item $[R, c] = -[R, d]$;
		
		\item $t_{R, c, \alpha}^{-1} = t_{R, d, -\alpha}$.
		
		Hint: Use Exercise~\ref{Exercise: roots}.
	\end{enumerate}
\end{enumerate}

\bibliographystyle{amsalpha}
\bibliography{references}

@book {AB08,
	AUTHOR = {Abramenko, Peter and Brown, Kenneth S.},
	TITLE = {Buildings},
	SERIES = {Graduate Texts in Mathematics},
	VOLUME = {248},
	NOTE = {Theory and applications},
	PUBLISHER = {Springer, New York},
	YEAR = {2008},
	PAGES = {xxii+747},
	ISBN = {978-0-387-78834-0},
	MRCLASS = {20E42 (20F55 20J06 51E24 51F15)},
	MRNUMBER = {2439729},
	MRREVIEWER = {Ralf Koehl},
	DOI = {10.1007/978-0-387-78835-7},
	URL = {https://doi.org/10.1007/978-0-387-78835-7},
}

@book {Bo02,
	AUTHOR = {Bourbaki, Nicolas},
	TITLE = {Lie groups and {L}ie algebras. {C}hapters 4--6},
	SERIES = {Elements of Mathematics (Berlin)},
	NOTE = {Translated from the 1968 French original by Andrew Pressley},
	PUBLISHER = {Springer-Verlag, Berlin},
	YEAR = {2002},
	PAGES = {xii+300},
	ISBN = {3-540-42650-7},
	MRCLASS = {17-01 (00A05 20E42 20F55 22-01)},
	MRNUMBER = {1890629},
	DOI = {10.1007/978-3-540-89394-3},
	URL = {https://doi.org/10.1007/978-3-540-89394-3},
}

@article {BPR19,
	AUTHOR = {Baumgartner, U. and Parkinson, J. and Ramagge, J.},
	TITLE = {Scale and tidy subgroups for {W}eyl-transitive automorphism
	groups of buildings},
	JOURNAL = {J. Algebra},
	FJOURNAL = {Journal of Algebra},
	VOLUME = {520},
	YEAR = {2019},
	PAGES = {460--478},
	ISSN = {0021-8693},
	MRCLASS = {22D05 (20E36 20E42 20F65)},
	MRNUMBER = {3884174},
	MRREVIEWER = {Tom De Medts},
	DOI = {10.1016/j.jalgebra.2018.11.018},
	URL = {https://doi.org/10.1016/j.jalgebra.2018.11.018},
}

@article {BRW07,
	AUTHOR = {Baumgartner, Udo and R\'{e}my, Bertrand and Willis, George A.},
	TITLE = {Flat rank of automorphism groups of buildings},
	JOURNAL = {Transform. Groups},
	FJOURNAL = {Transformation Groups},
	VOLUME = {12},
	YEAR = {2007},
	NUMBER = {3},
	PAGES = {413--436},
	ISSN = {1083-4362},
	MRCLASS = {22F50 (20E42 22D05 22E67)},
	MRNUMBER = {2356316},
	MRREVIEWER = {Guy Rousseau},
	DOI = {10.1007/s00031-006-0050-3},
	URL = {https://doi-org.ezproxy.uni-giessen.de/10.1007/s00031-006-0050-3},
}

@book{HR79,
	address = {Berlin},
	author = {Hewitt, Edwin and Ross, Kenneth A.},
	edition = {Second},
	isbn = {3-540-09434-2},
	mrclass = {43-02 (22-01)},
	mrnumber = {551496 (81k:43001)},
	pages = {ix+519},
	publisher = {Springer-Verlag},
	series = {Grundlehren der Mathematischen Wissenschaften [Fundamental Principles of Mathematical Sciences]},
	title = {Abstract harmonic analysis. {V}ol. {I}},
	volume = {115},
	year = {1979}
}

@article {Mo02,
	AUTHOR = {M{\"{o}}ller, R\"{o}gnvaldur G.},
	TITLE = {Structure theory of totally disconnected locally compact
	groups via graphs and permutations},
	JOURNAL = {Canad. J. Math.},
	FJOURNAL = {Canadian Journal of Mathematics. Journal Canadien de
	Math\'{e}matiques},
	VOLUME = {54},
	YEAR = {2002},
	NUMBER = {4},
	PAGES = {795--827},
	ISSN = {0008-414X},
	MRCLASS = {22D05 (05C25 20B99)},
	MRNUMBER = {1913920},
	MRREVIEWER = {Helge Gl\"{o}ckner},
	DOI = {10.4153/CJM-2002-031-5},
	URL = {https://doi.org/10.4153/CJM-2002-031-5},
}

@book {We03,
	AUTHOR = {Weiss, Richard M.},
	TITLE = {The structure of spherical buildings},
	PUBLISHER = {Princeton University Press, Princeton, NJ},
	YEAR = {2003},
	PAGES = {xiv+135},
	ISBN = {0-691-11733-0},
	MRCLASS = {51E24 (20E42 20F55)},
	MRNUMBER = {2034361},
	MRREVIEWER = {Theo Grundh\"{o}fer},
}

@book {We09,
	AUTHOR = {Weiss, Richard M.},
	TITLE = {The structure of affine buildings},
	SERIES = {Annals of Mathematics Studies},
	VOLUME = {168},
	PUBLISHER = {Princeton University Press, Princeton, NJ},
	YEAR = {2009},
	PAGES = {xii+368},
	ISBN = {978-0-691-13881-7},
	MRCLASS = {51E24 (20E42)},
	MRNUMBER = {2468338},
	MRREVIEWER = {Theo Grundh\"{o}fer},
}

@article {Wi04,
	AUTHOR = {Willis, George A.},
	TITLE = {Tidy subgroups for commuting automorphisms of totally
	disconnected groups: an analogue of simultaneous
	triangularisation of matrices},
	JOURNAL = {New York J. Math.},
	FJOURNAL = {New York Journal of Mathematics},
	VOLUME = {10},
	YEAR = {2004},
	PAGES = {1--35},
	MRCLASS = {20E25 (22D05 22E25)},
	MRNUMBER = {2052362},
	MRREVIEWER = {Volker Runde},
	URL = {http://nyjm.albany.edu:8000/j/2004/10_1.html},
}

@misc{willis2025flatgroups,
	title={Flat groups of automorphisms of totally disconnected, locally compact groups}, 
	author={George A. Willis},
	year={2025},
	eprint={2512.10509},
	archivePrefix={arXiv},
	primaryClass={math.GR},
	url={https://arxiv.org/abs/2512.10509}, 
}

\end{document}